\documentclass[preprint]{imsart}

\RequirePackage[OT1]{fontenc}
\RequirePackage{amsthm,amsmath,mathrsfs}
\RequirePackage[numbers]{natbib}
\RequirePackage[colorlinks,citecolor=blue,urlcolor=blue]{hyperref}
\usepackage[ansinew]{inputenc}
\usepackage{nicefrac,dsfont,bbm,amstext,amsfonts,amssymb,verbatim}
\usepackage{floatrow,rotating}
\usepackage{multicol}
\usepackage{multirow,tabularx,lscape}
\usepackage{booktabs}
\usepackage{changes}
\graphicspath{{Figures/}}
\startlocaldefs
\numberwithin{equation}{section}
\theoremstyle{plain}
\newtheorem{thm}{Theorem}[section]

\newtheorem{prop}[thm]{Proposition}
\newtheorem{lem}[thm]{Lemma}
\newtheorem{cor}[thm]{Corollary}
\newtheorem{assump}[thm]{Assumption}
\theoremstyle{remark}

\newcommand{\oo}{{\scriptstyle{\mathcal{O}}}}
\renewcommand{\O}{\mathcal{O}}
\renewcommand{\P}{\mathbb{P}}
\newcommand{\E}{\mathbb{E}}
\newcommand{\N}{\mathds{N}}
\newcommand{\R}{\mathds{R}}

\newcommand{\Ai}{\text{Ai}}
\newcommand{\Gi}{\text{Gi}}
\newcommand{\AI}{\text{AI}}
\newcommand{\h}{b}
\newcommand{\ad}{\mathfrak{a}}
\newcommand{\cd}{\lambda}
\newcommand{\OO}{\mathcal{O}}

\newcommand{\KK}{\mathcal{K}}

\providecommand{\abs}[1]{\lvert #1 \rvert}

\providecommand{\babs}[1]{{\Bigl\lvert #1 \Bigr\rvert}}

\providecommand{\eps}{\varepsilon}
\renewcommand{\epsilon}{\varepsilon}

\newcommand{\1}{\mathbbm{1}}
\DeclareSymbolFont{largesymbols}{OMX}{yhex}{m}{n}
\DeclareMathAccent{\verywidehat}{\mathord}{largesymbols}{'144}

\newcommand{\var}{\mathbb{V}\hspace*{-0.05cm}\textnormal{a\hspace*{0.02cm}r}}
\newcommand{\Var}{\mathbb{V}\hspace*{-0.05cm}\textnormal{a\hspace*{0.02cm}r}}

 \DeclareMathOperator{\Exp}{Exp}
 \DeclareMathOperator{\Poiss}{Poiss}

\newcommand{\KLEINO}{{\scriptstyle{\mathcal{O}}}}

\def\Links{\tagsleft@true}\def\Rechts{\tagsleft@false}
\makeatother
\def\lsim{\mathrel{\rlap{\lower4pt\hbox{\hskip1pt$\sim$}}
    \raise1pt\hbox{$<$}}}                % less than or approx. symbol
\def\gsim{\mathrel{\rlap{\lower4pt\hbox{\hskip1pt$\sim$}}
    \raise1pt\hbox{$>$}}}
\endlocaldefs

\renewcommand{\phi}{\varphi}

\renewcommand{\theta}{\vartheta}
\renewcommand{\subset}{\subseteq}

\newcommand{\F}{\mathcal{F}}

\allowdisplaybreaks[3]
\begin{document}

\begin{frontmatter}
%\title{Improved volatility estimation based on intra-day order book quotes}
%\runtitle{Microsturcture noise in a two price economy}
%\thankstext{T1}{Footnote to the title with the ``thankstext'' command.}
\title{Volatility estimation under one-sided errors with applications to limit order books \thanksref{T1}
%\title{Improved volatility estimation under one-sided errors with an application to limit order books \thanksref{T1}
}

\thankstext{T1}{Financial support from the Deutsche Forschungsgemeinschaft via SFB 649 {\it \"Okonomisches Risiko} and FOR 1735 {\it Structural Inference in Statistics: Adaptation and Efficiency} is gratefully acknowledged. The authors are grateful for helpful comments by the referee.}
\begin{aug}
\author{\fnms{Markus} \snm{Bibinger}%{bibinger@math.hu-berlin.de}},
},
\author{\fnms{Moritz} \snm{Jirak}%{jirak@math.hu-berlin.de}
},
\author{\fnms{Markus} \snm{Rei\ss}%{mreiss@math.hu-berlin.de}
}

%\ead[label=u1,url]{http://www.foo.com}}

\runauthor{M. Bibinger, M. Jirak \& M. Rei{\ss}}

\affiliation{Universität Mannheim and Humboldt-Universit\"at zu Berlin}

\address{Markus Bibinger,\\
Department of Economics,\\
Mannheim University,\\
L7, 3-5, 68161 Mannheim}
\address{Moritz Jirak,\\
Markus Rei\ss,\\
Institut f\"ur Mathematik\\
Humboldt-Universit\"at zu Berlin\\
Unter den Linden 6\\
10099 Berlin, Germany}

\end{aug}
%%%%%%%%%%%%%%%%%%%%%%%%%%%%%%%%%%%%%%%%%%%%%%%%%%%%%%%%%%%%%%%%%%%%%%%%%%%%%%%%%%%%%%%%%%%%%%%%%%%%%%%%%%%%%%%%%%%%%%%%%%%%%%%%%%%%%%%%%%%%%%%%%%%%%%%%%%%%%%%%%%%%%%%%%%%%%%%%%
\begin{abstract}
For a semi-martingale $X_t$, which forms a stochastic boundary, a rate-optimal estimator for its quadratic variation $\langle X, X \rangle_t$ is constructed based on observations in the vicinity of $X_t$. The problem is embedded in a Poisson point process framework, which reveals an interesting connection to the theory of Brownian excursion areas. We  derive $n^{-1/3}$ as optimal convergence rate in a high-frequency framework with $n$ observations (in mean). We discuss a potential application for the estimation of the integrated squared volatility of an efficient price process $X_t$ from intra-day order book quotes. %An estimator based on local order statistics attaining the rate is presented.
\end{abstract}

%60G55  	Point processes 60H10  	Stochastic ordinary differential equations 60H15  	Stochastic partial differential equations [See also 35R60]
%60H20  	Stochastic integral equations 	60H30  	Applications of stochastic analysis
\end{frontmatter}

%%%%%%%%%%%%%%%%%%%%%%%%%%%%%%%%%%%%%%%%%%%%%%%%%%%%%%%%%%%%%%%%%%%%%%%%%%%%%%%%%%%%%%%%%%%%%%%%%%%%%%%%%%%%%%%%%%%%%%%%%%%%%%%%%%%%%%%%%%%%%%%%%%%%%%%%%%%%%%%%%%%%%%%%%%%%%%%%%

\section{Introduction\label{sec:intro}}
Consider observations $({\cal Y}_i)$ above a stochastic boundary $(X_t,\,t\in[0,1])$, which is formed by the graph of a continuous semi-martingale.
The objective is to optimally recover the driving characteristic $\langle X, X \rangle_t$ of the boundary $X_t$, given the observations $({\cal Y}_i)$.
A quantification of the information content in these observations is non-trivial and leads to intriguing mathematical questions. One motivation for considering this stochastic boundary problem stems from financial applications in the context of limit order books. From a microeconomic point of view ask prices will typically lie above the efficient market price.
Here the underlying latent efficient log-price of a stock $(X_t,\,t\in[0,1])$, observed over a trading period like a day, serves as the boundary, whereas ask prices form the observations $({\cal Y}_i)$. Bid prices can be handled symmetrically and independently, which can be used to validate the model.

Climate physics provides another example where semi-martingales appear as stochastic boundaries. Considerable efforts are devoted to understanding the driving stochastic term for SDEs modeling the long-term temperature evolution, see for instance \cite{imkeller_storch_2010} and \cite{majda_climate_models_2001}. One key source for historical temperature data is given by annual tree rings ({\it dendrochronology} and {\it dendroclimatology}, see e.g. \cite{diaz_tree}), whose relationship with temperature in an ideal environment is known.  For individual trees only sizes up to this ideal boundary are observed due to growth obstructions like limited nutrition, leading to deviations modeled as negative observation errors.

As a prototype model, we consider the continuous It\^{o} semi-martingale
 \begin{align}
\label{ito}X_t=X_0+\int_0^t a_s\,ds+\int_0^t\sigma_s\,dW_s\,,t\in[0,1],
\end{align}
with (possibly stochastic) drift and volatility coefficients $a_s$  and $\sigma_s$, defined on a filtered probability space $(\Omega,\mathcal{F},(\mathcal{F}_t),\P)$ with  a standard $({\mathcal F}_t)$-Brownian motion $W$. Its total quadratic variation $ \langle X, X \rangle_1=\int_0^1\sigma_s^2 \,ds$ is commonly called integrated squared volatility. Section \ref{sec:4} provides a generalization to models with jumps.
\begin{figure}[t]
\fbox{\includegraphics[width=6.2cm]{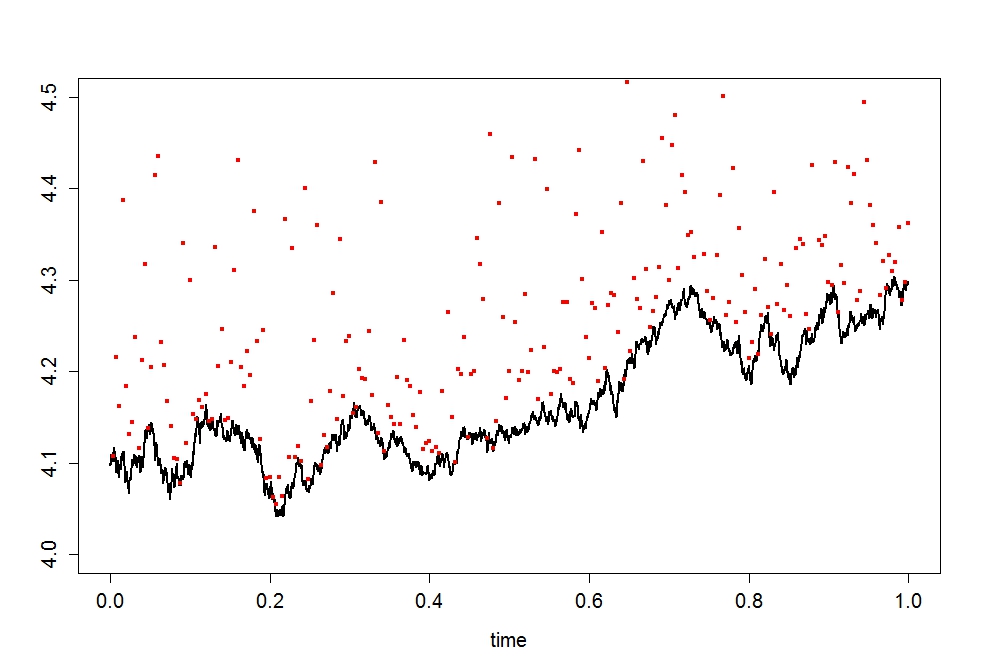}\includegraphics[width=6.2cm]{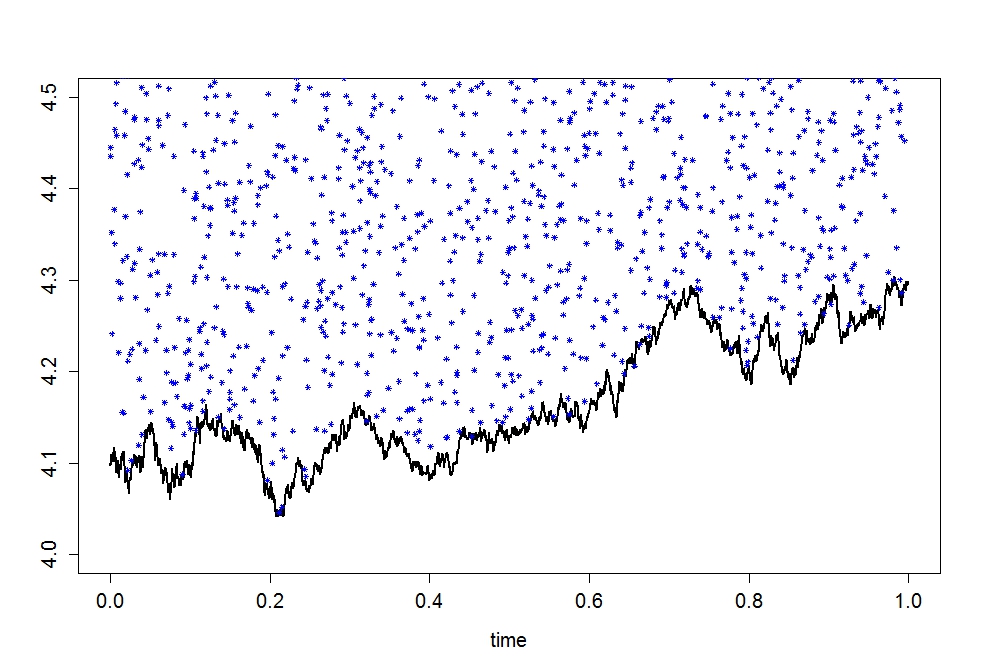}}
\caption{\label{Fig:1} Left: Microstructure noise model $Y_i=X_{i/n}+\eps_i,i=0,\ldots,n=1000$, with $\eps_i\stackrel{iid}{\sim}\text{Exp}(50)$.
Right: Poisson point process model with intensity $\lambda_{t,y}=50 n\1_{(y\ge X_t)}$ with $X_t$ an It\^{o} process.}
\end{figure}
A natural continuous-time embedding of the boundary problem is in terms of a Poisson point process (PPP). Conditional on $(X_t,\,t\in[0,1])$ we observe a PPP on $[0,1]\times\R$ with
intensity measure
\begin{align}\label{EqPPPint}
\Lambda(A) = \int_0^1\int_{\R} \1_A(t,y) \lambda_{t,y}\, dt\,dy,  \quad \text{where $\lambda_{t,y} = n\lambda\1(y\ge X_t)$.}
\end{align}
We denote by $(T_j,{\cal Y}_j)$ the observations of that point process, which are homogeneously dispersed above the graph of $(X_t,\,t\in[0,1])$, cf.\;Figure \ref{Fig:1}. Theoretically and also intuitively, information on the stochastic boundary can only be recovered from the lowest observation points  and a homogeneous intensity away from the boundary is assumed for convenience only.

An associated discrete-time regression-type model, which explains well the difference to classical noise models,
is defined by
\begin{align}\label{noiseuni}
Y_i=X_{t_i^n}+\eps_i\,,\,i=0,\ldots,n,\;\eps_i\ge 0,\;\eps_i\stackrel{iid}{\sim}F_{\lambda}\,,
\end{align}
with observation times $t_i^n$ and an error distribution function $F_\lambda$ satisfying
\begin{align}\label{defn_noise_distrib_property}
F_{\lambda}(x) =\lambda x\big(1+\KLEINO(1)\big) ,\;\mbox{as}~x\downarrow 0.
\end{align}
One natural parametric specification is  $\eps_i\sim\text{Exp}(\lambda)$, cf.\;Figure \ref{Fig:1}.
The noise is assumed to be independent of the signal part $X$. In microstructure noise models for transaction prices it is usually assumed that $\E[\eps_i]=0$ holds, while here $X_{t_i^n}$ determines the boundary of the support measure for $Y_i$. In fact, if the boundary function was piecewise constant, then by standard PPP properties we would obtain the regression-type model \eqref{noiseuni} with exponential noise from the PPP-model \eqref{EqPPPint} by taking local minima (on those pieces).
Here we show that under so called high-frequency asymptotics, the fundamental quantities in both models exhibit the same asymptotic behaviour, see Proposition \ref{prop_min_distrib_det} below. Compare also \cite{meisterreiss} for the stronger Le Cam\,--\,equivalence in the case of smoother boundaries.

We shall first concentrate on the more universal PPP model which also allows for  simpler scaling and geometric interpretation. Local minima $m_{n,k}$ of ${\cal Y}_j$ for $T_j$ in some small intervals $[kh_n,(k+1)h_n)\subset [0,1]$ will form the basic quantities to recover the boundary, which by PPP properties leads to the study of
\[ \P(m_{n,k}>x)=
 \E\left[\exp{\Big(-\int_{kh_n}^{(k+1)h_n} (X_t+x)_+\,dt\Big)}\right], \,x\in\R,
\]
where $A_+=\max(A,0)$, and its associated moments.  For the fundamental case $X_t=\sigma W_t$, this opens an interesting connection to the theory of Brownian excursion areas and also reveals the difficulty of this problem. It is well documented in the
literature, see e.g.\,\cite{janson}, that no explicit form of the expectation in the expression above is available. Essentially only (double) Laplace transforms and related quantities are known, cf.\,Proposition \ref{propfeynman} below and the attached discussion. This makes the recovery of $\langle X, X \rangle_1$ an intricate probabilistic question. Still, we are able to prove that our estimator attains the rate $n^{-1/3}$. What is more, by information-theoretic arguments we are able to derive a lower bound showing that the $n^{-1/3}$-rate is indeed minimax optimal. A more direct proof seems out of reach because the Poisson part from the noise intertwines with the Gaussian  martingale part in a way which renders the likelihood and respective Hellinger distances difficult to control, even asymptotically.

So far, the growing finance literature on limit order books focusses on modeling and empirical studies.
Empirical contributions as \cite{biais1}, \cite{bouchard1} and \cite{naes} have investigated price and volume distribution, inter-event durations as well as the structure of the order-flow.
Probabilistic models proposed for a limit order book include point process models, see \cite{cont1}, \cite{bacry} and \cite{rosenbaum}, with mutually exciting processes. Other models come from queuing theory, for instance \cite{foucault}, \cite{rosu} and \cite{cont2}, or stochastic optimal control theory as \cite{carmona}.
The main objective of most modeling approaches is to explain how market prices arise from the book.
\begin{figure}[t]
\includegraphics[width=10cm]{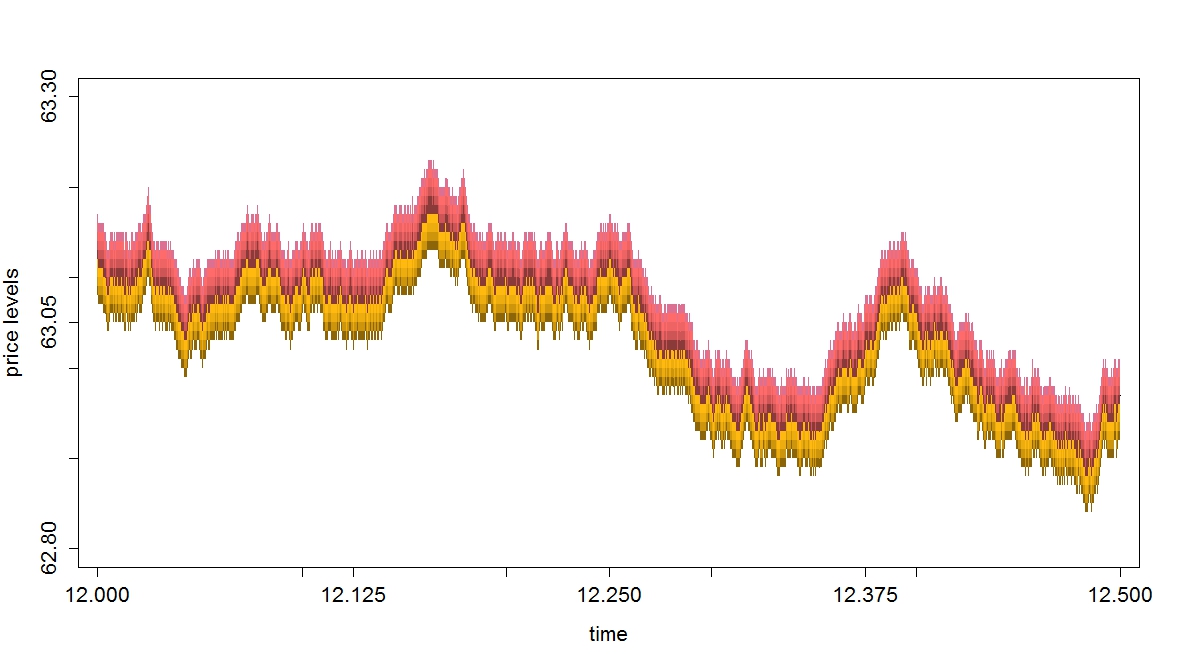}
\caption{\label{Fig:2}Order price levels for Facebook asset (NASDAQ) from 12:00 to 12:30 on June 2nd 2014. Colored areas highlight spreads between different bid and ask levels from level 1 up to level 5, bid-ask spread is colored in dark red.\protect\footnotemark}
\end{figure}
\footnotetext{Data provided by LOBSTER academic data -- powered by NASDAQ OMX.}
For the financial application, this papers adopts a different course  with the focus on estimating the latent volatility based on observations from a limit order book.
Contrarily, to the regular microstructure noise model which constitutes the standard setup for developing volatility estimators from transaction data, see e.g. \cite{howoften}, \cite{zhangmykland}, \cite{bn} and \cite{JPV}, among many others, our model assumes one-sided noise. The optimal convergence rate for volatility estimation in the model with Gaussian or regular centered noise and $n$ observations on an equidistant grid is $n^{-1/4}$, see \cite{gloter}.

Recently, as information from order books become more and more available, researchers and practitioners have sparked the discussion to which kind of observed prices estimation methods should be applied. \cite{micro} discuss this point and the possibilities of mid-quotes, executed traded prices or micro-prices which are volume-weighted combinations of bid and ask order levels. None of these observed time series, however, is free from market microstructure corruptions and the idea of an underlying efficient price remains untouched. Figure \ref{Fig:2} visualizes the information about the evolution of prices provided by a limit order book for one specific data set. The colored areas highlight differences between the five best bid and five best ask levels, the dark area in the center marking the bid-ask spread between best bid and best ask. The idea is that an efficient price should lie, at least most of the time, below the best ask (and symmetrically above the best bid) and that its distance to this stochastic frontier is homogeneous. Similar reasoning served as the fundament of the order book model by \cite{osterrieder} as well as for the dynamic trading model by \cite{yacine}.

Since modeling in science, economics and particularly finance is always a compromise between catching major features and too complex descriptive models, robustness to model misspecification is a key issue. Therefore we propose a simple modification of our estimator such that occasional violations of the continuous semi-martingale model do not change the asymptotic properties of the estimator. We shall show that for general violations, in particular evoked by jumps of the efficient price and the volatility our adjusted estimation method is robust.

The remainder of the paper is organized as follows. In Section \ref{sec:3} we present an estimation approach based on local order statistics whose asymptotic properties are explored in Section \ref{sec:4}, along with the robustification against violations. In Section \ref{sec:5} we prove the lower bound for the minimax estimation rate. An empirical example is performed in Section \ref{sec:6} which concludes with a discussion. Proofs are provided in the Appendix.

\section{Volatility estimation based on local minima\label{sec:3}}
We construct the integrated squared volatility estimator in both models \eqref{EqPPPint} and \eqref{noiseuni}. We partition the unit interval into $h_n^{-1}\in\mathds{N}$ equi-spaced bins $\mathcal{T}_k^n=[kh_n,(k+1)h_n),k=0,\ldots,h_n^{-1}-1,$ with bin-widths $h_n$. For simplicity suppose that $nh_n\in\N$.  As $n\rightarrow\infty$ the bin-width gets smaller $h_n\rightarrow 0$, whereas the number of observed values on each bin gets large, $nh_n\rightarrow\infty$. If we think of a constant signal locally on a bin observed with one-sided positive errors, classical parametric estimation theory motivates to use the bin-wise minimum as an estimator of the local signal (it then forms a sufficient statistic under exponential noise or equivalently in the PPP model). In the regression-type model \eqref{noiseuni} with equidistant observation times $t_i^n=i/n$, we therefore set
\begin{align}\label{localmin}m_{n,k}=\min_{i\in\mathcal{I}_k^n}Y_i~,~\mathcal{I}_k^n=\{kh_nn,kh_nn+1,\ldots,(k+1)h_nn-1\}\,.\end{align}
Equally, in the PPP model \eqref{EqPPPint} the local minima are given by
\begin{align}\label{localminPPP}m_{n,k}=\min_{T_j\in\mathcal{T}_k^n}{\cal Y}_j~,~\mathcal{T}_k^n=[kh_n,(k+1)h_n)\,.\end{align}
The same symbol $m_{n,k}$ is used in both models because the following construction only depends on the $m_{n,k}$. All results and proofs will refer to the concrete model under consideration.

Since $\Var(m_{n,k}\,|\,(X_t))\propto (n\lambda h_n)^{-2}$ holds in both models, the variance is much smaller than for an estimator based on a local mean. Nevertheless, we may continue in the spirit of the pre-averaging paradigm, cf.\;\cite{JPV}, and interpret $m_{n,k}$ as a proxy for $X_t$ on ${\cal T}_k^n$, which in a second step is inserted in the realized variance expression $\sum_{i=1}^{h_n^{-1}}(X_{kh_n}-X_{(k-1)h_n})^2$ without noise. The use of a locally constant signal approximation $X_t=X_{kh_n}+\O_{\P}(h_n^{1/2})$ on ${\cal T}_k^n$ is only admissible, however, if $h_n$ is chosen so small that $h_n^{1/2}=o((n\lambda h_n)^{-1})$, which would result in a sub-optimal procedure.

\begin{figure}[t]
\fbox{\includegraphics[width=6.2cm]{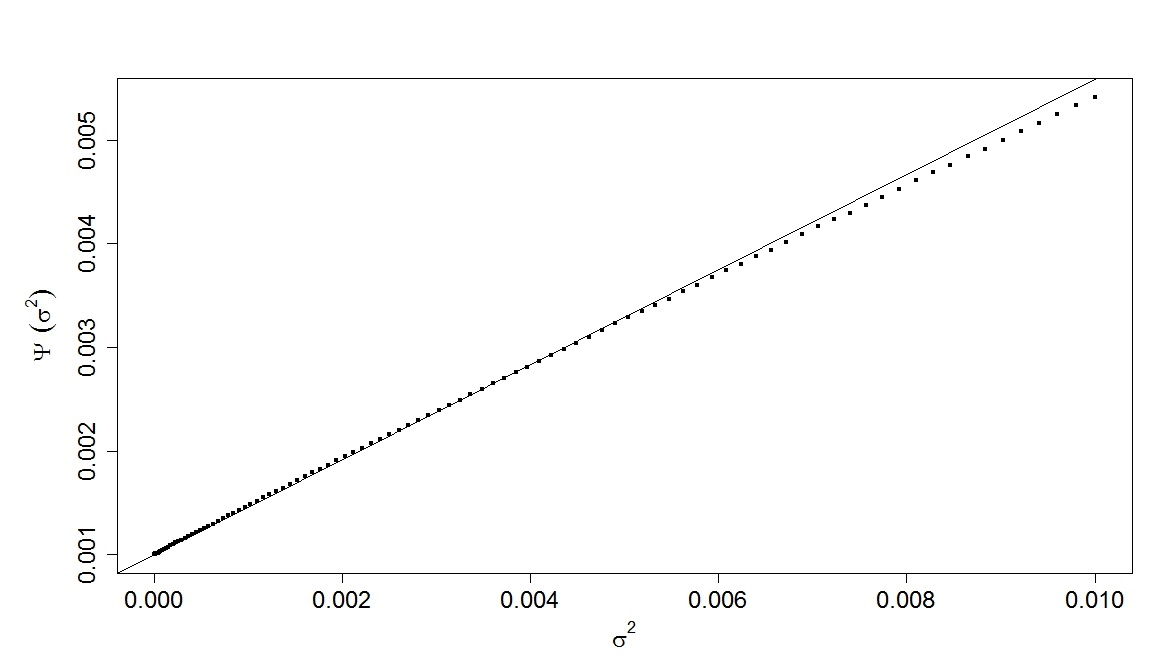}\includegraphics[width=6.2cm]{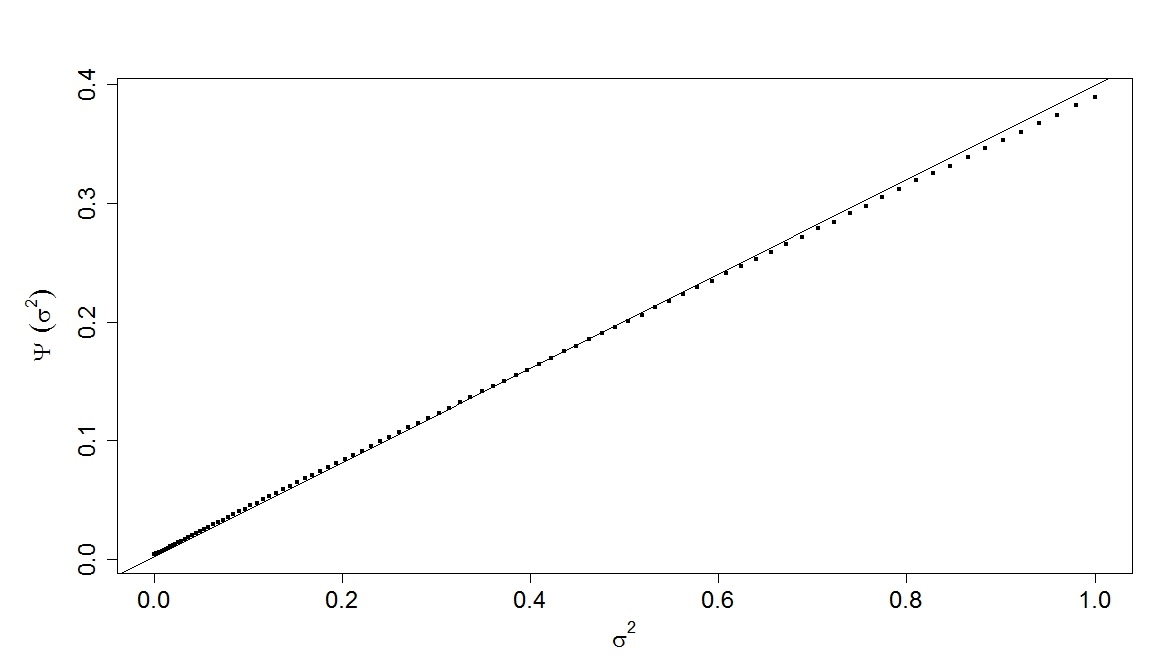}}
\caption{\label{Fig:3} The points indicate the function $\Psi(\sigma^2)$ with $\KK=31.6$ for small (left) and moderate (right) values of $\sigma^2$. The calculation is based on accurate Monte Carlo simulations. The lines show close linear functions for comparison. }
\end{figure}

Rate-optimality can be attained if we balance the magnitude $(n\lambda h_n)^{-1}$ of bin-wise minimal errors due to noise with the range $h_n^{1/2}$ of the motion of $X$ on the bin. This gives the order
\begin{align}\label{order}h_n\propto (n\lambda)^{-\frac23},~nh_n\propto n^{\frac13}\lambda^{-\frac23}~.\end{align}
In the PPP model \eqref{EqPPPint} this natural choice of the bin-width also follows nicely by a scaling argument: $\bar W_t= h_n^{-1/2}W_{h_nt}$ defines a standard Brownian motion for $t\in[0,1]$ based on the values of $W$ on $[0,h_n]$; the correspondingly scaled PPP observations $(\bar T_j,\bar{\cal Y}_j)$ with $\bar T_j=h_n^{-1}T_j$, $\bar{\cal Y}_j=h_n^{-1/2}{\cal Y}_j$ have an intensity with density $\bar\lambda_{t,y}=n\lambda h_n^{3/2}\1(y\ge \bar W_t)$, which becomes independent of $n$ exactly for $h_n=(n\lambda)^{-2/3}$.

In this balanced setup the law of the statistics $m_{n,k}$ depends on the motion of $X$ as well as the error distribution in a non-trivial way.
Still, the natural statistics to assess the quadratic variation of the boundary process $X$ are the squared differences $(m_{n,k}-m_{n,k-1})^2$ between consecutive local minima. In the PPP model and with the choice \begin{equation}\label{Eqhn}
h_n=\KK^{\frac{2}{3}} (n\lambda)^{-\frac{2}{3}}\text{ for some constant } \KK>0
\end{equation}
the law of $h_n^{-1/2}m_{n,k}$ is independent of $n$, $h_n$ and $\lambda$ and for $X_t=X_{(k-1)h_n}+\sigma\int_{(k-1)h_n}^t\,dW_s$ on ${\cal T}_{k-1}^n\cup{\cal T}_k^n$, we may introduce
\begin{align}\label{psiPPP}
\Psi\big(\sigma^2\big)=h_n^{-1}\E\big[\big(m_{n,k}-m_{n,k-1}\big)^2\big]\,,k=1,\ldots,h_n^{-1}-1.
\end{align}

Below we shall derive theoretical properties of $\Psi$ and in particular we shall see that it is invertible as soon as $\KK>0$ is chosen sufficiently large. Numerically, the function $\Psi$ can be determined by standard Monte Carlo simulations, see Figure \ref{Fig:3}, and is thus available. This paves the way for a moment-estimator approach. In fact, $\sum_k(m_{n,2k}-m_{n,2k-1})^2$ approximates $\int \Psi(\sigma_t^2)dt$ with corresponding summation and integration intervals. Under regularity assumptions on $t\mapsto \sigma_t^2$ and by the smoothness of $\Psi$ shown below, we have
\begin{align} \Psi^{-1}\Bigg(\sum_{k=(l-1)r_n^{-1}/2+1}^{ lr_n^{-1}/2} \hspace*{-.05cm}\big(m_{n,2k}\hspace*{-.05cm}-\hspace*{-.05cm}m_{n,2k-1}\big)^2\,2h_n^{-1}r_n\Bigg)\approx \sigma_{lr_n^{-1}h_n}^2,\end{align}
where $r_n^{-1}h_n$ is a coarse grid size with $r_nh_n^{-1}\in\N,r_n^{-1}\in 2\N$. This gives rise to the following estimator of integrated squared volatility $IV=\int_0^1\sigma_t^2dt$ in the PPP model \eqref{EqPPPint} with bin-width \eqref{Eqhn}:
\begin{align}\label{estimatorPPP}\widetilde{IV}_n^{h_n,r_n} \hspace*{-.1cm}= \hspace*{-.1cm}\sum_{l=1}^{r_nh_n^{-1}} \hspace*{-.1cm}\Psi^{-1}\hspace*{-.05cm}\Bigg(\sum_{k=(l-1)r_n^{-1}/2+1}^{ lr_n^{-1}/2} \hspace*{-.05cm}\big(m_{n,2k}\hspace*{-.05cm}-\hspace*{-.05cm}m_{n,2k-1}\big)^2\,2h_n^{-1}r_n\Bigg)h_nr_n^{-1}.\end{align}

In the regression-type model \eqref{noiseuni} the corresponding second moments still depend on $n$ and we write explicitly
\begin{align}\label{psireg}
\Psi_n\big(\sigma^2\big)=h_n^{-1}\E\big[\big(m_{n,k}-m_{n,k-1}\big)^2\big]\,,k=1,\ldots,h_n^{-1}-1.
\end{align}
We shall see below that $\Psi_n\to\Psi$ holds, but a non-asymptotic form of the volatility estimator from regression-type observations is  given by
\begin{align}\label{estimator}\widehat{IV}_n^{h_n,r_n} \hspace*{-.1cm}= \hspace*{-.1cm}\sum_{l=1}^{r_nh_n^{-1}} \hspace*{-.1cm}\Psi_n^{-1}\hspace*{-.05cm}\Bigg(\sum_{k=(l-1)r_n^{-1}/2+1}^{ lr_n^{-1}/2} \hspace*{-.05cm}\big(m_{n,2k}\hspace*{-.05cm}-\hspace*{-.05cm}m_{n,2k-1}\big)^2\,2h_n^{-1}r_n\Bigg)h_nr_n^{-1}.
\end{align}

For a parametric estimation of $\sigma_t=\sigma=\text{const.}$, we employ the global moment-type estimator $\widehat{IV}_n^{h_n,h_n}$. Here, inversion of the entire sum of squared differences is conducted. In the nonparametric case of varying $\sigma_t$ instead a localized estimator $\widehat{IV}_n^{h_n,r_n}$, with $r_n\rightarrow 0, r_n^{-1}h_n\rightarrow 0$, is applied. A balance between a second order term on each coarse interval of order $r_n$ and an approximation error controlled by a semi-martingale assumption on $\sigma_t$ of order $r_n^{-1}h_n$ will lead to the choice $r_n\propto h_n^{1/2}\propto (n\lambda)^{-1/3}$.

\section{The convergence rate of the estimator%Properties of the estimator
\label{sec:4}}

%\subsection{The law of local minima and asymptotic properties of the estimator\label{sec:4.1}}

\begin{figure}[t]
\fbox{\includegraphics[width=10cm]{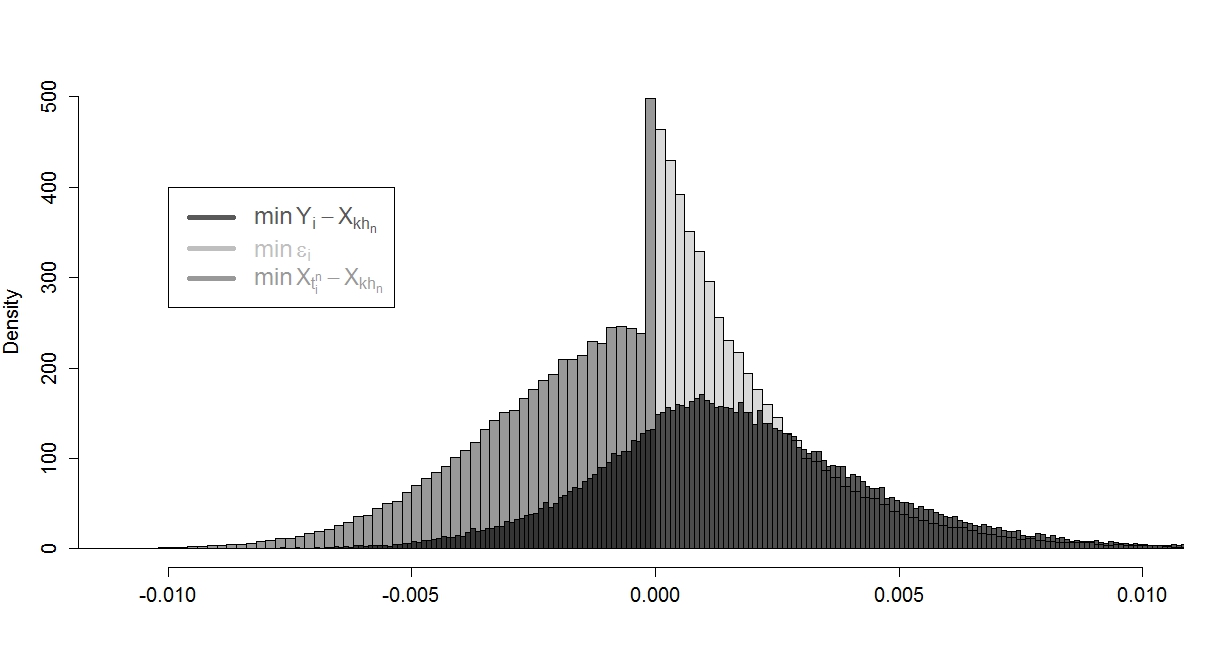}}
\caption{\label{Fig:4}Distributions of bin-wise minima of the signal process, noise and the convolution. Based on 100000 simulated bins with $\sigma=1$, $\eps_i\sim\Exp(5)$, $nh_n=100$.}
\end{figure}

In order to centralize the local minima, we write
\begin{align}\label{rewrite}m_{n,k}-m_{n,k-1}=\mathcal{R}_{n,k}-\mathcal{L}_{n,k}~,k=1,\ldots,h_n^{-1}-1\,,\end{align}
where $\mathcal{R}_{n,k}=m_{n,k}-X_{kh_n}$ and $\mathcal{L}_{n,k}=m_{n,k-1}-X_{kh_n}$ measure the distances between the minima on bin ${\cal T}_k^n$ and ${\cal T}_{k-1}^n$, respectively, to the central true value $X_{kh_n}$ between both bins.
In our high-frequency framework the drift is asymptotically negligible and a regular volatility function will  be approximated by a piecewise constant function on blocks of the coarse grid. In this setting, where $X_t=X_{kh_n}+\sigma(W_t-W_{kh_n})$ and $\sigma$ is deterministic, we may invoke time-reversibility of Brownian motion to see that $X_t-X_{kh_n}$, $t\in{\cal T}_{k-1}^n$, and $X_t-X_{kh_n}$, $t\in{\cal T}_k^n$, form independent Brownian motions of variance $\sigma^2$ such that $\mathcal{R}_{n,k},\mathcal{L}_{n,k},k= (l-1)r_n^{-1}+1,\ldots, lr_n^{-1}$, are all identically distributed and there is independence whenever different bins are considered (but $\mathcal{R}_{n,k}$ and $\mathcal{L}_{n,k+1}$ are dependent).
From \eqref{psiPPP} and \eqref{rewrite} we infer
\[\Psi(\sigma^2_{kh_n})h_n=\E[\mathcal{R}_{n,k}^2]+\E[\mathcal{L}_{n,k}^2]-2\E[\mathcal{R}_{n,k}]\E[\mathcal{L}_{n,k}]
=2\,\var(\mathcal{R}_{n,k}),
\]
and similarly for $\Psi_n$. The histogram in Figure \ref{Fig:4} shows the distribution of $\mathcal{R}_{n,k}$ (equivalently $\mathcal{L}_{n,k}$) in the regression model jointly with the associated histograms for $\min_{i\in\mathcal{I}_k^n}X_{t_i^n}-X_{kh_n}$ and $\min_{i\in\mathcal{I}_k^n}\epsilon_i$. In this situation the law of $\mathcal{R}_{n,k}$ is given as the convolution between an exponential distribution and the law of the minimum of Brownian motion on the discrete grid ${\cal I}_k^n$. The latter converges to the law of the minimum of $W$ on $[0,1]$, but the simulations confirm the known feature that the laws deviate rather strongly around zero for moderate discretisations. Let us state and prove a slightly more general result.

\begin{prop}\label{prop_min_distrib_det_PPP}
Choose $h_n$ according to \eqref{Eqhn}. Consider $t\in{\cal T}_k^n$ for fixed $k$ and suppose that $X_t=X_{kh_n}+\int_{kh_n}^t \sigma \,dW_s, t\in\mathcal{T}_k^n$. Then in the PPP model \eqref{EqPPPint} for all $x\in\R$
\begin{align}
\P\Big(h_n^{-1/2}{\cal R}_{n,k}>x\sigma\Big) = \E\Big[\exp\Big(-\KK\sigma\int_0^1 (x + W_t )_+ \,dt \Big)\Big].
\end{align}
\end{prop}

\begin{proof}
By conditioning on the Brownian motion we infer from the PPP properties of $(T_j,{\cal Y}_j)$:
\begin{align*}
\P\Big(h_n^{-1/2}{\cal R}_{n,k}>x\sigma\,\big|\,W\Big)&=\exp\Bigg(-\int_{{\cal T}_k^n}\int_{-\infty}^{x\sigma h_n^{1/2}+X_{kh_n}}\lambda_{t,y}\,dt\,dy\Bigg)\\
&=\exp\Bigg(-n\lambda\sigma\int_{{\cal T}_k^n}\big(xh_n^{1/2}-(W_t-W_{kh_n})\big)_+\,dt\Bigg).
\end{align*}
Noting that $\bar W_s=h_n^{-1/2}(W_{(k+s)h_n}-W_{kh_n})$, $s\in[0,1]$, is again a Brownian motion, the result follows by rescaling and taking expectations.
\end{proof}

For the regression-type model the survival function is asymptotically of the same form.

\begin{prop}\label{prop_min_distrib_det}
Choose $h_n$ according to \eqref{Eqhn}.  Suppose that $X_t=X_{kh_n}+\int_{kh_n}^t \sigma \,dW_s$, $t\in {\cal T}_k^n$, for a fixed bin number $k$. Then in the regression-type model \eqref{noiseuni} for all $x\in\R$
\begin{align}\label{exp}
\lim_{n\to\infty}\P\Big(h_n^{-1/2}\mathcal{R}_{n,k}>x \sigma\Big)= \E\Big[\exp{\Big(-\KK \sigma\int_0^1(x+W_t)_+\,dt\Big)}\Big].
\end{align}
\end{prop}

The approximation error due to non-constant $\sigma$ and drift is considered in detail in Appendix \ref{theouppera} and proved to be asymptotically negligible.
This way, the asymptotic analysis of our estimation problem leads  into the theory of Brownian excursion areas. Let  $\mathcal{R}_t$ be a real random variable distributed as $\lim_{n\rightarrow\infty} h_n^{-1/2}\mathcal{R}_{n,\lfloor th_n^{-1}\rfloor}$. The law of $\mathcal{R}_t$ determines $\Psi(\sigma_t^2)$ via
\begin{align}\var(\mathcal{R}_t)=\frac12  \Psi(\sigma_t^2)\,.\end{align}
The Feynman--Kac formula gives a connection of the right-hand side in Proposition \eqref{prop_min_distrib_det_PPP} to a parabolic PDE based on the heat semigroup for Brownian motion. We can prove the following explicit result on the Laplace transform which determines the distribution of $(\mathcal{R}_t), {t\in[0,1]}$.

\begin{prop}\label{propfeynman}
The Laplace transform (in $t$) of $$\E\Big[\exp\Big(-\sqrt{2}\theta\int_0^t(x+W_s)_+\,ds\Big)\Big]$$
with $\theta\in\R$ satisfies the following identity:
\begin{align*}
\E\left[\int_0^{\infty}\exp{\Big(-st-\sqrt{2}\theta\int_0^t(x+W_s)_+\,ds\Big)}\,dt\right]=\theta^{-\frac23}\zeta_s(x,\theta),
\end{align*}
with $\zeta_s(x,\theta)=\zeta_{s,-}(x,\theta)\1_{(-\infty,0)}(x)+\zeta_{s,+}(x,\theta)\1_{[0,\infty)}(x)$ defined by the functions
%\begin{subequations}
\begin{align*}\zeta_{s,+}(x,\theta)&=\frac{\pi\big(\theta^{{1/3}}\Gi^{\,\prime}\big(\theta^{{-2/3}}s\big)-\sqrt{s}\Gi\big(\theta^{-2/3}s\big)\big)+\theta^{2/3}s^{-1/2}}{\sqrt{s}\Ai\big(\theta^{-2/3}s\big)-\theta^{{1/3}}\Ai^{\,\prime}\big(\theta^{-2/3}s\big)}\\ &\quad\times \Ai\big(\sqrt{2}\theta^{1/3}x+\theta^{-2/3}s\big)+\pi\Gi\big(\sqrt{2}\theta^{1/3}x+\theta^{-2/3}s\big)\,,\\
\zeta_{s,-}(x,\theta)&=\left(\frac{\theta^{2/3}s^{-1/2}\Ai\big(\theta^{{-2/3}}s\big)+\theta^{1/3}\AI\big(\theta^{{-2/3}}s\big)}{\sqrt{s}\Ai\big(\theta^{{-2/3}}s\big)-\theta^{{1/3}}\Ai^{\,\prime}\big(\theta^{{-2/3}}s\big)}-s^{-1}\theta^{2/3}\right)\\
&\quad \times \exp{\big(\sqrt{2s}x)}+s^{-1}\theta^{2/3}\,,\end{align*}
%\end{subequations}
where $\Ai$ is the Airy function which is bounded on the positive half axis,
$$\Ai(x)=\pi^{-1}\int_0^{\infty}\cos{(t^3/3+xt)}\,dt\,,$$
and $\Gi$ is the Scorer function bounded on the positive half axis
$$\Gi(x)=\pi^{-1}\int_0^{\infty}\sin{(t^3/3+xt)}\,dt\,,$$
and we define $\AI(x)=\int_x^{\infty}Ai(y)dy.$
\end{prop}
This result generalizes the Laplace transform of the exponential integrated positive part of a Brownian motion derived by \cite{perman}. Inserting $x=0$ and setting $\theta=1$ renders the result by \cite{perman}. An inversion of the Laplace transform in Proposition \ref{propfeynman} in order to obtain an explicit form of the distribution function and then $\Psi$ appears unfeasible as several experts vainly attempted to solve related problems, see \cite{perman} and \cite{janson}. Exploiting the strong Markov property of Brownian motion together with
hitting times, we are able to circumvent this problem in our study of $\Psi(\sigma^2)$, for details we refer to the Appendix.

The  observation models \eqref{EqPPPint} and \eqref{noiseuni} as well as the semi-martingale model \eqref{ito} for $X$ might be idealized. In finance, effects of surprise elements and information processing might occasionally result in violations of this  model, for instance by price jumps. In such situations, a regularization of $\Psi^{-1}(\cdot)$ can yield more robust estimation results. We propose to truncate the estimator on the coarse grid by employing
\begin{align}\label{truncation} \Psi_{\tau}^{-1}(\cdot) = \Psi^{-1}(\cdot) \wedge \tau\end{align}
for some $\tau > 0$ instead of $\Psi^{-1}$, giving the adjusted estimators $\widetilde{IV}_{n,\tau}^{h_n,r_n}$ and $\widehat{IV}_{n,\tau}^{h_n,r_n}$  in \eqref{estimatorPPP} and \eqref{estimator}. The truncation level $\tau>0$ is chosen such that we can guarantee $\sup_{t \in [0,1]\setminus \mathcal{V}_n}\sigma_t^2 \leq \tau$ almost surely where $\mathcal{V}_n$ denotes the union of all violated blocks. In practice, any over-estimated bound from independent historical data may work. Furthermore, observe that any continuous process of finite variation $\bar A_t$ may corrupt the observations via $Y_i=X_{t_i^n}+A_{t_i^n}+\eps_i$ without harming our volatility estimator because it can be incorporated as a drift into the new semi-martingale $X_t+A_t$. For order books, the corruption $A_{t_i^n}$ may account for spreads  due to market processing and inventory costs.

We  formulate now the main convergence results whose proofs are given in the Appendix. For that we impose some mild regularity on the drift and diffusion coefficient. Moreover, we need that the function $\Psi$ is invertible and sufficiently regular, which by Proposition \ref{thm_psi_properties} below is ensured by a sufficiently large choice of $\KK$, but at least numerically seems to be the case for much smaller choices, cf.\;Figure \ref{Fig:3}. % and \cite{BJRSim}.
We work under the general structural hypothesis that the volatility is an It\^{o} semi-martingale with finite activity jumps. This is a standard assumption in financial volatility estimation, see e.g.\,\cite{bn} and \cite{JPV}, allowing for stochastic volatility with leverage. To remain concise, we assume global conditions on the characteristics, but extensions via localization techniques as in Section 4.4.1 of \cite{JP} are clearly possible.

\begin{assump}\label{sigma} We work in the stochastic volatility model with potential jumps in $X$ and $\sigma$:
 %\begin{align}
%\sigma_t = \sigma_0 + \int_0^t \widetilde{a}_s \,ds + \int_0^t \widetilde{\sigma}_s\, d W_s + \int_0^t \widetilde{\eta}_s \,d W_s^{\bot}
%\end{align}
\begin{align*}%\label{ito_jumps}
X_t &= X_0 + \int_0^t a_s\,ds+\int_0^t\sigma_s\,dW_s + \int_0^t \int_{\R} x \, d\mu^{X}\bigl(ds, dx \bigr),\\
\sigma_t &= \sigma_0 + \int_0^t \tilde a_s\,ds+\int_0^t\tilde\sigma_s\,dW_s+\int_0^t\tilde\eta_s\,dW_s^\perp + \int_0^t \int_{\R} x \, d\mu^{\sigma}\bigl(ds, dx \bigr).
\end{align*}
with finite random measures $\mu^{{X}},\mu^{\sigma}$, i.e. $(\mu^{{X}}+\mu^{\sigma})( [0,1], \R )<\infty$ almost surely.
Assume that the volatility is uniformly bounded away from zero, i.e; $\inf_{0 \leq t \leq 1} \sigma_t \geq \sigma_-$ almost surely for a deterministic constant $\sigma_->0$. The characteristics $a_s$, $\widetilde{a}_s$, $\widetilde{\sigma}_s$ and $\widetilde{\eta}_s$ are progressively-measurable and uniformly bounded. The constant $\KK$ in the definition \eqref{Eqhn} of $h_n$ is chosen large enough that Proposition \ref{thm_psi_properties} below applies.
\end{assump}

\begin{thm}\label{theoupper}
Grant Assumption \ref{sigma} with continuous $X$ from \eqref{ito}, choose $h_n$ according to \eqref{Eqhn} and $r_n=\kappa n^{-1/3}$ for some $\kappa>0$. Then the estimator \eqref{estimatorPPP} based on observations from the PPP-model satisfies
\begin{align}\Big(\widetilde{IV}_{n}^{h_n,r_n}-\int_0^1\sigma_s^2\,ds\Big)=\mathcal{O}_{\P}\big(n^{-\frac13}\big)\,.\end{align}
On the general Assumption \ref{sigma}, utilizing the truncation in \eqref{truncation} and if $\sup_{t \in [0,1]}\sigma_t^2 \leq \tau$ almost surely, it also holds that
\begin{align}\Big(\widetilde{IV}_{n,\tau}^{h_n,r_n}-\int_0^1\sigma_s^2\,ds\Big)=\mathcal{O}_{\P}\big(n^{-\frac13}\big)\,.\end{align}
\end{thm}
Based on the same strategy of proof we can obtain an analogous result for the regression-type model.
\begin{cor}\label{corrupper}
Grant Assumption \ref{sigma}, choose $h_n$ according to \eqref{Eqhn} and $r_n=\kappa n^{-1/3}$ for some $\kappa>0$. Then the estimator \eqref{estimator} based on observations from the regression-type model satisfies the same asymptotic properties as estimator \eqref{estimatorPPP} in Theorem \ref{theoupper}.
%\begin{align}\Big({\color{blue}{\widehat{IV}_{n,\tau}^{h_n,r_n}}}-\int_0^1\sigma_s^2\,ds\Big)=\mathcal{O}_{\P}\big(n^{-\frac13}\big)\,.\end{align}
\end{cor}

In the rate-optimal balanced setup there are three error contributions of the same order: the implied observational noise on the bins $[kh_n,(k+1)h_n)$, the bin-wise approximation of $X$, and a second order term on the coarse blocks arising from the nonlinearity of $\Psi$. Their interplay is non-trivial and thus a general stable central limit theorem for the rescaled error does not seem straight-forward. If we dropped the ambition of rate-optimality, however, we could undersmooth or oversmooth by a different choice of the block sizes $h_n$ and $r_n$ such that only one or two error terms would prevail for which estimators with a simpler asymptotic distribution theory  would be available. This is not pursued here.

\section{Lower bound for the rate of convergence\label{sec:5}}
Consider our PPP-model  \eqref{EqPPPint}. We show that even in the simpler parametric statistical experiment where
$X_t=\sigma W_t, \,t\in[0,1]$, and $\sigma>0$ is unknown the optimal rate of convergence is $n^{-1/3}$ in a minimax sense. This lower bound for the parametric case then serves a fortiori as a lower bound  for the general nonparametric case. A lower bound for the discrete regression-type model is obtained in a similar way; in fact the proof is even simpler, replacing the Poisson sampling $(T_j^s)$ below by a deterministic design of distance $n^{-2/3}$.

\begin{thm}\label{theolower}
We have for any sequence of estimators $\hat\sigma_n^2$ of $\sigma^2\in(0,\infty)$
from the parametric PPP-model for each $\sigma_0^2>0$, the local minimax lower
bound
\[ \exists
\delta>0:\;\liminf_{n\to\infty}\inf_{\hat\sigma_n}\max_{\sigma^2\in\{\sigma_0^2,\sigma_0^2+\delta
n^{-1/3}\}}
\P_{\sigma^2}(\abs{\hat\sigma_n^2-\sigma^2}\ge  \delta n^{-1/3})>0,
\]
where the infimum extends over all estimators $\hat\sigma_n$ based on the  PPP-model  \eqref{EqPPPint} with $\lambda=1$ and $X_t=\sigma W_t$. The law of the latter is denoted by $\P_{\sigma^2}$.
\end{thm}

The proof falls into three main parts. We first simplify the problem by considering more informative experiments. These reductions are given in the two steps below. Then, in the third step we use bounds for the Hellinger distance. The more technical step 3 is worked out in Appendix \ref{prooflb}.
\begin{enumerate}
\item A PPP with intensity $\Lambda$ is obtained as the sum of two independent PPPs
with intensities $\Lambda_r$ and $\Lambda_s$, respectively, satisfying
$\Lambda=\Lambda_r+\Lambda_s$, see e.g.\,\cite{karr}. Hence, for $\h>0$
the experiment of observing $(T_i^r,\mathcal{Y}_i^r)_{i\ge 1}$ from a PPP with regularised
intensity density
\[\lambda_r(t,y)=n\Big(\big((y-X_t)_+/\h\big)^2\wedge 1\Big)\]
and independently
$(T_j^s,\mathcal{Y}_j^s)_{j\ge 1}$ from a PPP with discontinuous intensity density
$\lambda_s=\lambda-\lambda_r$ is more informative. We now provide even more
information by replacing $(T_j^s,\mathcal{Y}_j^s)_{j\ge 1}$ by $(T_j^s,X_{T_j^s})_{j\ge 1}$,
the direct observation of the martingale values at the random times $(T_j^s)$. A
lower bound proved for observing $(T_i^r,\mathcal{Y}_i^r)_{i\ge 1}$ and
$(T_j^s,X_{T_j^s})_{j\ge 1}$ independently thus also applies to the original (less informative)
observations.

\item Due to $\int\int \lambda_s(t,y)dt\,dy=(2/3)n\h$, we conclude that the times $(T_j^s)$ are
given by a Poisson sampling of intensity $(2/3)n\h$ on $[0,1]$ and there are a.s.\,only
finitely many times $(T_j^s)_{j=1,\ldots,J}$. Let us first work conditionally on
$(T_j^s)$ and put $T_0^s=0$, $T_{J+1}^s=1$. All observations of $(T_i^r,\mathcal{Y}_i^r)_{i\ge
1}$ with $T_i^r\in[T_{j-1}^s,T_j^s)$ are transformed via
    \[(T_i^r,\mathcal{Y}_i^r)\mapsto
\Bigg(T_i^r-T_{j-1}^s,\mathcal{Y}_i^r-\Bigg(X_{T_{j-1}^s}\frac{T_i^r-T_{j-1}^s}{T_j^s-T_{j-1}^s}+X_{T_{j}^s}\frac{T_{j}^s-T_i^r}{T_j^s-T_{j-1}^s}\Bigg)
\Bigg).\]
    Noting that $(B_t-(t/T)B_T, t\in[0,T])$ defines a Brownian bridge $B^{0,T}$ on
$[0,T]$, we thus obtain conditionally on $(T_j^s)$ for each $j=1,\ldots,J+1$
observations of a PPP on $\big[0,T_j^s-T_{j-1}^s\big]$ with intensity density
\[\lambda^j(t,y)=n\Big(\h^{-1}\big(y-\sigma B^{0,T_j^s-T_{j-1}^s}_t\big)_+\wedge 1\Big).\]
The transformation has rendered the family
of PPPs with intensity densities $(\lambda^j)_{j=1,\ldots,J+1}$ independent by reducing
the Brownian motion to piecewise Brownian bridges. Conditionally on $(T_j^s)$ we
thus have independent observations of $(T_j^s,X_{T_j^s})_{j=1,\ldots,J}$ and
independent PPPs with intensity densities $(\lambda^j)_{j=1,\ldots,J+1}$.
\end{enumerate}

By using the latter more informative experiment and by choosing $\h\propto n^{-1/3}$ we show below that for a Poisson sampling $(T_j^s)_{j=1,\ldots,J}$ on $[0,1]$ of intensity $(2/3)n\h\propto n^{2/3}$ of direct observations $X_{T_j^s}$ as well as for independent observations of PPPs, generated by $\sigma$ times a Brownian bridge in-between the sampling points $(T_j^s)_j$, we cannot estimate at a better rate than $n^{-1/3}$. This is accomplished by bounding the Hellinger distance between the experiments for $\sigma^2=\sigma_0^2$ and $\sigma^2=\sigma_0^2+\delta n^{-1/3}$.

%\section{{\color{blue}{An empirical example\label{sec:6}}}}
\section{Discussion\label{sec:6}}
\begin{figure}[t]
\caption{\label{Fig:5}Estimated daily volatilities multiplied by $10^5$ for Facebook, August 2015. Days on which jumps have been filtered are highlighted by $*$-symbols.}

\centering
\frame{
\includegraphics[width=\textwidth]{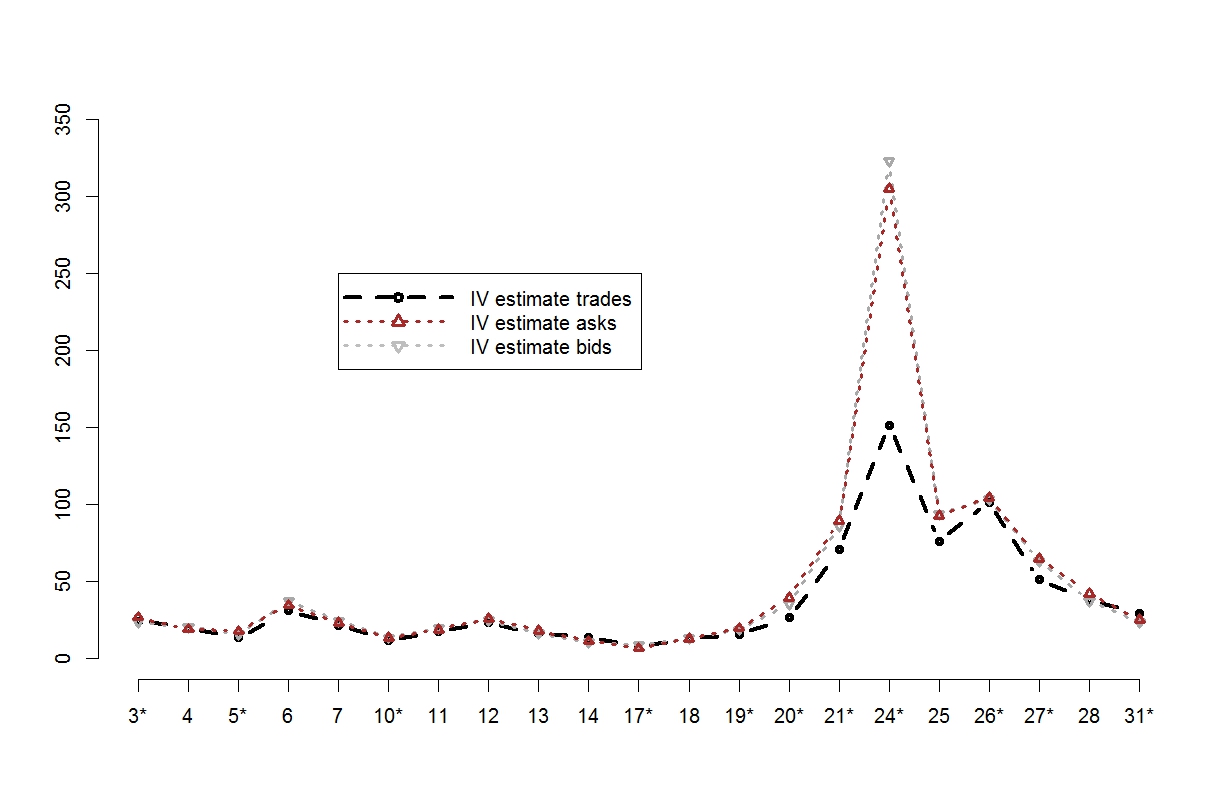}}
\begin{comment}{\footnotesize \textit{Notes: The points mark the estimates, rescaled with factor $10^5$, for each trading day based on observations of bid quotes, ask quotes and traded log-prices, respectively. The latter are quantified using the efficient local method of moments estimator from \cite{BHMR} for the regular noise model with centered observation errors. Trading days on which jumps occurred and have been truncated in the process are highlighted by $*$-symbols added to the dates on the x-axis.}}
\end{comment}

\end{figure}

For the application to limit order books we model the relationship between ask quotes and an efficient price process by a pure boundary model, not taking into account the fine structure of order book dynamics. This agnostic point of view seems attractive for statistical purposes because more complex models will usually require more data for the same estimation accuracy and are highly exposed to model misspecification. To check whether the semi-martingale boundary model leads to realistic results, we apply our estimator to limit order book data and compare it to integrated volatility estimators, which are commonly used for traded prices under market microstructure noise.

We consider limit order book data of the Facebook asset (FB) traded at NASDAQ provided by LOBSTER academic data, recorded over the 21 trading days in August 2015. Empirical data analysis with similar assets lead to comparable results. The August 2015 time series has the advantage of starting with a relatively calm period before incorporating a period of high trading activity, which can serve as a kind of stress test to the estimators.  We estimate day-wise integrated volatilities based on
\begin{enumerate}
\item Our estimator \eqref{estimator} with truncation for the regression-type model and first level ask quotes log-prices (and symmetrically, but independently bid quotes), called $\widehat{IV}$ in the sequel. The average number of newly submitted best ask quotes per day, $n$, in the considered period is about 100,000. The average $n$ for bid quotes is similar, but the difference on certain days may be large. The maximal absolute difference is 30,537 in the considered period.
\item The local method of moment (LMM) estimator from \cite{BHMR}, adjusted to possible jumps with the truncated version from \cite{bibwinkel}, and log-prices from trades reconstructed from the order book. The average number of trades per day in the considered period is about 43,000.
\end{enumerate}
The LMM is the asymptotically efficient estimator for the standard regular noise model with centered noise and we follow the implementation with a selection of tuning parameter described in \cite{BHMR2}. For the truncation step we employ a global threshold $\tau=2\log(h_n^{-1})h_n \widehat{ IV}^{pre}$, with the pre-estimator of integrated volatility obtained from the first estimation step of the two-stage adaptive LMM, for both approaches using their \emph{different} bin-lengths $h_n$ and sample sizes $n$. $h_n$ is chosen in a data-driven way, for the LMM we arrive at about 100 bins per day and for the $\widehat{IV}$ at about 650 bins. We expect that comparative studies using alternative estimators for the regular noise model, as e.g.\,realized kernels from \cite{bn} or pre-averaging from \cite{JPV}, would yield similar results.

The results are presented in Figure \ref{Fig:5}. It shows a rather close relationship between the three sequences of estimates. Estimates obtained from bid and ask quotes can differ, but their differences are very small.
On August 24, 2015, however,  there is a large difference among the estimators. On that day a flash crash manipulated the traded and order quote prices challenging all market models. The huge difference among the truncated estimates is due to the rougher time resolution of bins for LMM which are  equidistant in calendar time. The flash crash led to tremendous price movements in very short time at the beginning of the trading day along with a huge trading activity. In the regular noise model the LMM ascribes those movements on its first two time bins to jumps and truncates, while our bin widths $h_n$ are much smaller and thus $\widehat{IV}$ is still affected by this period because not all bins are truncated. The same effect explains why the values of $\widehat{IV}$ are significantly larger around that date.

A priori, even from a microeconomic perspective, it is not clear whether the assumed efficient price processes for the different market mechanisms giving rise to bids, asks and trades are the same or at least exhibit the same integrated volatility. A proper statistical test for the latter hypothesis requires a simultaneous distribution theory for $\widehat{IV}$ and the LMM (a CLT alone is not sufficient), which is beyond the scope of the present work and a project in its own right. %at least an asymptotic central-limit-type theorem with empirically accessible variance, which does not seem easily available given the complicated error structure.
For the LMM alone, however, a feasible central limit theorem is available, see Theorem 4.4 of \cite{BHMR}.

For the Facebook data set we have conducted a test on the hypothesis that the integrated volatilities in the order book and transaction price models coincide, assuming independence of the estimators and a Gaussian limit distribution where the variance of $\widehat{IV}$ does not exceed the one of LMM. Applied to 21 trading days and at asymptotic level $\alpha=5\%$ , the test has accepted the null on 14 days and rejected on 7 days. This testing problem illustrates that more mathematical analysis of the estimator's risk is highly desirable as well as a more profound empirical study.

\appendix

\section{Proofs of Section 3}
Proposition \ref{prop_min_distrib_det} considers the simplified model where $X_t,t\in{\cal T}_k^n$, is approximated by $X_{kh_n}+\int_{kh_n}^t\sigma_{kh_n}\,dW_t$. The resulting approximation error is bounded within Proposition \ref{prop_min_distrib_det_PPP2} for the PPP-model and an analogous proof carries over to the regression-type model. In the sequel, we write $A_+=A\1(A\ge 0)$, $A_-=|A|\1{(A\le 0)}$ and $\|Z\|_p=\E[|Z|^p]^{1/p}, p\ge 1$.

\begin{proof}[Proof of Proposition \ref{prop_min_distrib_det}]
By law invariance of ${\cal R}_{n,k}$ with respect to $k$ for $X_t=X_0+\sigma W_t$, we can simplify
\begin{align*}
\P(h_n^{-1/2}{\cal R}_{n,k}>x\sigma)&=\P\Big(h_n^{-1/2}\min_{i=0,\ldots,nh_n-1}(X_{i/n}-X_0+\eps_i)>x\sigma\Big)\\
&=\P\Big(\min_{i=0,\ldots,nh_n-1}( W_{i/(nh_n)}+\sigma^{-1}h_n^{-1/2}\eps_i)>x\Big),
\end{align*}
where we used that $h_n^{1/2}W_{t/h_n}$ is another Brownian motion. We condition on the driving Brownian motion $W=(W_t,t\in[0,1])$ and obtain in terms of the distribution function $F_\lambda$ of $\eps_i$:
\begin{align*}
\P(h_n^{-1/2}{\cal R}_{n,k}>x\sigma) &= \E\Bigg[\prod_{i=0}^{nh_n-1}\P\big(\eps_i>\sigma h_n^{1/2}(x-W_{i/(nh_n)})\,\big|\,W\big)\Bigg]\\
&=\E\Big[\exp\Big(\sum_{i=0}^{nh_n-1}\log\Big(1-F_\lambda(\sigma h_n^{1/2}(x-W_{i/(nh_n)}))\Big)\Big)\Big].
\end{align*}
The expansion \eqref{defn_noise_distrib_property} of $F_\lambda$ together with expanding the logarithm therefore yields
\[ \P(h_n^{-1/2}{\cal R}_{n,k}>x\sigma)=\E\Big[\exp\Big(-\sigma h_n^{1/2}\lambda\sum_{i=0}^{nh_n-1}(x-W_{i/(nh_n)})_+(1+\oo(1))\Big)\Big],\]
where $\oo(1)$ is to be understood $\omega$-wise and holds uniformly over $i$ and $n$ whenever $\max_{t\in[0,1]}(x-W_{t}(\omega))_+$ is bounded. By the choice of $h_n$ we have $h_n^{1/2}\lambda=\KK(nh_n)^{-1}$ and the integrand is a Riemann sum tending almost surely to $\exp(-\sigma\KK\int_0^1(x-W_t)_+dt)$. Noting that a conditional probability is always bounded by 1, the assertion follows by dominated convergence and use of $-W\stackrel{d}{=}W$.
\end{proof}

%\MR ONLY NEED: $F_\lambda(x)=\lambda x(1+\oo(1))$ as $x\downarrow 0$!

\begin{proof}[Proof of Proposition \ref{propfeynman}]
Throughout the proof, we drop the dependence on $\theta$ in $\zeta_s(x,\theta)$, $\zeta_{s,-}(x,\theta)$ and $\zeta_{s,+}(x,\theta)$ to lighten the notation. We shall apply the Kac formula in the version as in formulae (4.13) and (4.14) of \cite{karatzas}. It connects the considered Laplace transform with the solution of a differential equation which becomes in our case:
\begin{subequations}
\begin{align}\label{psi-ode}\frac{d^2\zeta}{dx^2}&=2s\zeta-2\theta^{2/3}&,\;x<0,~~~~~~~~~~~~~~~~~~~~~~~~~~\\
\label{psi+ode}\frac{d^2\zeta}{dx^2}&=2(\sqrt{2}\theta x+s)\zeta-2\theta^{2/3}&,\;x>0.~~~~~~~~~~~~~~~~~~~~~~~~~~\end{align}
\end{subequations}
Since all assertions necessary to apply the Kac formula are fulfilled, the Laplace transform from above multiplied with a constant Lagrangian $\theta^{2/3}$ satisfies
$$\E\left[\int_0^{\infty}\theta^{2/3}\exp\Big(-st-\sqrt{2}\theta\int_0^t(x+W_s)_+\,ds\Big)\,dt\right]=\zeta_s(x)\,.$$
The general solution of \eqref{psi-ode} is given by
\begin{subequations}
\begin{align}\label{psi-g}\zeta_{s,-}(x)=A\exp{\big(\sqrt{2s}x\big)}+\theta^{2/3}s^{-1}\,,\end{align}
with a constant $A$ (depending on $s$ but not on $x$). Airy's function $\Ai$ solves the homogenous differential equation of the type \eqref{psi+ode}, whereas the Scorer function $\Gi$ is a particular solution of the inhomogenous equation $\zeta''-x\zeta=\pi^{-1}$, both being bounded on the positive real line. Hence,  a solution ansatz for \eqref{psi+ode} is given by
\begin{align}\label{psi+g}\zeta_{s,+}(x)=B\Ai\big(\sqrt{2}\theta^{1/3}x+\theta^{-2/3}s\big)+\pi\Gi\big(\sqrt{2}\theta^{1/3}x+\theta^{-2/3}s\big)\,,\end{align}
\end{subequations}
with a constant $B$. Continuity conditions on $\zeta$ and $d\zeta/dx$ at $x=0$ give rise to
\begin{align*}B=\frac{\pi\big(\theta^{{1/3}}\Gi^{\,\prime}\big(\theta^{{-2/3}}s\big)-\sqrt{s}\Gi\big(\theta^{-2/3}s\big)\big)+\theta^{2/3}s^{-1/2}}{\sqrt{s}\Ai\big(\theta^{-2/3}s\big)-\theta^{{1/3}}\Ai^{\,\prime}\big(\theta^{-2/3}s\big)}\,.\end{align*}
In order to express $A$ in a more concise and simple manner, we exploit the following relation for the Wronskian of $\Ai$ and $\Gi$:
\begin{align}\label{wronskian}\pi\Big(\Gi^{\,\prime}(x)\Ai(x)-\Ai^{\,\prime}(x)\Gi(x)\Big)=\AI(x)=\int_x^{\infty}\Ai(y)\,dy\,.\end{align}
A proof of the latter equality can be found in \cite{airy}. Thereby, we obtain
\begin{align*}A=\left(\frac{\theta^{2/3}s^{-1/2}\Ai\big(\theta^{{-2/3}}s\big)+\theta^{1/3}\AI\big(\theta^{{-2/3}}s\big)}{\sqrt{s}\Ai\big(\theta^{{-2/3}}s\big)-\theta^{{1/3}}\Ai^{\,\prime}\big(\theta^{{-2/3}}s\big)}-s^{-1}\theta^{2/3}\right)\,.\end{align*}
This result concludes the proof.
\end{proof}

%\section{new}

\subsection{Asymptotic analysis of the estimator}  \label{theouppera}

Recall that due to Assumption \ref{sigma} we can assume without loss of generality that $\|a\|_{\infty}, \|\widetilde{a}\|_{\infty}, \|\widetilde{\sigma}\|_{\infty}, \|\widetilde{\eta}\|_{\infty} \leq  C_{+}$  a.s., and that $\inf_{0 \leq t \leq 1} \sigma_t \geq \sigma_- > 0$ a.s. Here, $C_+$ and $\sigma_-$ are absolute constants. From here on $A_n\lesssim B_n$ expresses shortly that $A_n\le K\cdot B_n$ for two sequences $A_n,B_n$ and some real constant $K<\infty$. We use the notation $A \lesssim_{\,a.s.} B$ if this holds $\P$-almost surely. Similarly, we write $A = \OO_{a.s.}(B)$ %if $A \lesssim_{\,a.s.} B$
and use $A\le_{a.s.}B$ for short notation. We also write $\P_{|k}(\cdot)=\P(\cdot | \F_{kh_n})$ and analogously for the conditional expectation. Moreover, we use $\|_{|k}Z\|_q=\E_{|k}[|Z|^q]^{1/q}, q\ge 1$.\\
First, we establish Theorem \ref{theoupper} and Corollary \ref{corrupper} on Assumption \ref{sigma} and in \emph{absence of jumps} in $X$ and $\sigma$ for estimators \eqref{estimatorPPP} and \eqref{estimator}, respectively. Robustness of the truncated versions against violations is proved at the end of this section.
As a first step, we analyze the approximation error assuming a locally constant volatility and
neglecting the drift. Then we prove Theorem \ref{theoupper} exploiting properties of $\Psi$ which are established in Appendix \ref{proppsiapp}. We shall use the following identities for moments of real random variables:
\begin{subequations}
\begin{align}\label{eq_exp_gen_1}
\E\bigl[X\bigr] = \int_0^{\infty}\P\bigl(X > x\bigr)dx - \int_{0}^{\infty}\P\bigl(-X > x \bigr)dx\,,
\end{align}

\begin{align}\label{eq_exp_gen_2}
\E\bigl[X^2\bigr] = 2\int_0^{\infty}x \P\bigl(X > x\bigr)dx + 2\int_{0}^{\infty}x\P\bigl(-X > x \bigr)dx\,.
\end{align}
\end{subequations}

\begin{lem}\label{lem_tech_sigma_squre}
For any finite $p > 1$ and $x > 0$, $0 \leq s \leq t \leq 1$, we have
\begin{align*}
%&\text{{\bf (i)}}\quad \P_{}\bigl(\sup_{s \leq r \leq t}|\sigma_r|\geq x |\F_s \bigr) \lesssim \bigl(|t|/x^2\bigr)^{p/2},
%\\
%&\text{{\bf (ii)}}
\quad \P_{}\bigl(\sup_{s \leq r \leq t}|\sigma_r^2 - \sigma_s^2|\geq x \sigma_s|\F_s \bigr) \lesssim_{\,a.s.} \bigl(|t-s|/(x^2 \wedge x)\bigr)^{p/2}.
\end{align*}
\end{lem}

\begin{proof}[Proof of Lemma \ref{lem_tech_sigma_squre}]
Since $\|\widetilde{a}\|_{\infty} \leq C_+$ a.s; we get that $\int_{s}^t |\widetilde{a}_r| d\,r \leq C_+(t-s)$ almost surely. Hence using Markov's and Burkholder's inequality, for any finite $p > 1$
\begin{align}\label{eq_lem_tech_sigma_squre_1}
\P_{}\bigl(\sup_{s \leq r \leq t}|\sigma_r - \sigma_s|\geq z |\F_s \bigr) \lesssim_{\,a.s.} \bigl(|t-s|/z^2\bigr)^{p/2}, \quad \text{$z > 0$,}
\end{align}
where we also used $\|\widetilde{\sigma}\|_{\infty}, \|\widetilde{\eta}\|_{\infty} \leq C_+$ a.s. Since
\begin{align*}
\sigma_r^2 - \sigma_s^2 = (\sigma_r - \sigma_s)^2 + 2\sigma_s(\sigma_r - \sigma_s),
\end{align*}
we obtain that
\begin{align*}
\P_{}\bigl(\sup_{s \leq r \leq t}|\sigma_r^2 - \sigma_s^2|\geq x \sigma_s|\F_s \bigr) &\leq \P_{}\bigl(\sup_{s \leq r \leq t}|\sigma_r - \sigma_s|\geq \sqrt{x \sigma_s /2} |\F_s \bigr) \\&\quad + \P_{}\bigl(\sup_{s \leq r \leq t}|\sigma_r - \sigma_s|\geq x/4 |\F_s \bigr).
\end{align*}
Hence using \eqref{eq_lem_tech_sigma_squre_1} and $\inf_{0 \leq s \leq 1}\sigma_s \geq \sigma_-$ a.s., the claim follows.
\end{proof}

\begin{prop}\label{prop_min_distrib_det_PPP2}
Consider $h_n$ in \eqref{Eqhn} and $t \in \mathcal{T}_k^n$ for fixed $k$. Then
\begin{align*}
&\P\left(\min_{j\in\mathcal{T}_k^n}\mathcal{Y}_j-X_{kh_n}> x \sigma_{k h_n} \sqrt{h_n}\bigl| \F_{kh_n}\right) \\&= \E\Big[\exp\Big(-\KK\sigma_{k h_n}\int_0^1 (x + \widetilde{W}_t )_+ dt \Big)\bigl| \F_{kh_n}\Big] + (\cd n)^{-1/3}  G(x), \quad a.s.,
\end{align*}
where $(\widetilde{W}_t)_{0 \leq t \leq 1}$ is a standard Brownian motion independent of $(\F_t)_{0 \leq t \leq 1}$, $G(x)$ is deterministic and $|x|^p|G(x)| \in L^1(\R)$ for any finite $p \geq 0$. If $\sigma_t$ is constant and $a_t = 0$ for $t \in \mathcal{T}_k^n$, then $G(x) = 0$.
\end{prop}

\begin{proof}[Proof of Proposition \ref{prop_min_distrib_det_PPP2}]
Throughout the proof, $C_0$, $C_1$, $\ldots$ denote positive, generic constants that may vary from line to line. Proposition \ref{prop_min_distrib_det_PPP} already gives the last statement in case of no drift and bin-wise constant volatility. Let $\mathcal{A}_z = \mathcal{T}_k^n \times (-\infty,z]$
%with $\mathcal{T}_k^n = [kh_n,(k+1)h_n)$
and
\begin{align}z = x \sigma_{k h_n} \sqrt{h_n}\,.\end{align}
Let $\Delta X_{t}(k) = \int_{kh_n}^t \sigma_s \,dW_s$ and $\Delta A_{t}(k) = \int_{kh_n}^t a_s \,ds$. Then, using basic properties of a PPP, it follows that
\begin{align}\label{eq_derive_PPP_1}\nonumber
\P_{|k}\left(\min_{j\in\mathcal{T}_k^n}\mathcal{Y}_j - X_{kh_n}>z\right)&= \E_{|k}\left[\P\left( \Lambda\bigl(\mathcal{A}_z\bigr) = 0 \big|X\right)\right] = \E_{|k}\Big[\exp\Big(-\Lambda(\mathcal{A}_z)\Big)\Big] \\& \nonumber\hspace*{-.5cm} = \E_{|k}\Big[\exp\Big(-n\cd\int_{\mathcal{A}_z} \1_{\{\Delta X_{t}(k) + \Delta A_{t}(k) \le y \}}\,dt\, dy\Big)\Big] \\& \hspace*{-.5cm}= \E_{|k}\Big[\exp\Big(-n\cd\int_{\mathcal{T}_k^n}(z-\Delta X_{t}(k) - \Delta A_t(k))_+\,dt\Big)\Big].
\end{align}
Introduce
\begin{align*}
T_k &= n\cd\int_{\mathcal{T}_k^n}(z-\Delta X_{t}(k) - \Delta A_t(k))_+dt\,,\\
V_k &=  n\cd\int_{kh_n}^{(k+1)h_n}\bigl(z - \sigma_{k h_n}(W_{t} - W_{kh_n})\bigr)_+dt\,,\\
U_k &= n\cd\int_{kh_n}^{(k+1)h_n} \biggl| \int_{k h_n}^t (\sigma_s -\sigma_{kh_n}) d W_s \biggr|dt ~\text{and} ~ A_k = n \cd h_n^2 \max_{t \in \mathcal{T}_k^n}|a_t |.
\end{align*}
Then we have the upper and lower bounds
\begin{align*}
V_k - U_k - A_k\leq T_k \leq V_k + U_k + A_k.
\end{align*}
By scaling and symmetry properties of Brownian motion, we have that
\begin{align*}
V_k \stackrel{d}{=} n \cd \sigma_{k h_n} h_n^{3/2} \int_0^1 (\widetilde{W}_t + x )_+ \,dt = \KK\sigma_{k h_n}\int_0^1(\widetilde{W}_t+x)_+\,dt\,,
\end{align*}
with a standard Brownian motion $\widetilde W$ independent of $(\F_t)_{0 \leq t \leq 1}$.
In the sequel, we distinguish the two cases where $x \geq -1$ and $x < -1$.\\
{\bf Case $x \geq -1$:} As a first objective, we derive an upper bound for $\E_{|k}\Big[\exp\Big(-yT_k\Big)\Big]$, $y > 0$. To this end, note that by the Dambis-Dubins-Schwarz Theorem (Thm. 4.6 in \cite{karatzas}), on a possibly larger probability space (extending time and processes from $T = 1$ to $T = \infty$), there exists a Brownian motion $(W_t^{\diamond})$ independent of $\F_{kh_n}$ such that
\begin{align*}
\Delta X_{t}(k) \stackrel{d}{=} W^{\diamond}_{\langle \Delta X(k), \Delta X(k) \rangle_t}.
\end{align*}
Lemma \ref{lem_tech_sigma_squre} yields that for $x \in \R$
\begin{align*}
\P_{|k}\left(\sup_{kh_n \leq t \leq (k+1)h_n}\int_{kh_n}^{t} |\sigma_s^2 - \sigma_{kh_n}^2|ds \geq |x| \sigma_{kh_n}\right) \lesssim_{\,a.s.} (h_n/|x|)^{p'/2}\wedge 1, \quad p' > 0.
\end{align*}
Since $\langle \Delta X(k), \Delta X(k) \rangle_t = \int_{k h_n}^t \sigma_s^2 \, ds$ for $t \geq k h_n$, we deduce
\begin{align}\label{eq_bound_X_k_1} %\nonumber
&\P_{|k}\hspace*{-.01cm}\Big(\sup_{t \in \mathcal{T}_k^n}|\Delta X_{t}(k)| \geq |z|/2 \Big) \hspace*{-.05cm}\lesssim_{\,a.s.} \hspace*{-.05cm}\P_{|k}\hspace*{-.05cm}\Big(\sup_{0 \leq t \leq |x| \sigma_{k h_n}} \hspace*{-.5cm} |W^{\diamond}_{\sqrt{t}}| \geq |z|/2 \Big) + (h_n/|x|)^{p'/2}\wedge 1
\\& \nonumber \lesssim_{\,a.s.} \hspace*{-.05cm}\P_{|k}\hspace*{-.05cm}\left(\sup_{0 \leq t \leq 1}\hspace*{-.05cm}|W_t^{\diamond}| \hspace*{-.05cm}\geq \hspace*{-.05cm}\frac{C_0\sigma_{k h_n} |x|}{\sigma_{k h_n} \hspace*{-.075cm}+ \hspace*{-.075cm}\sqrt{|x| \sigma_{k h_n}}}\hspace*{-.05cm} \right) + (h_n/|x|)^{p'/2}\wedge 1
 \\& \nonumber \hspace*{.5cm} \lesssim_{\,a.s.} \P\bigl(|W_1^{\diamond}| \geq C_0 \sqrt{|x|}  \bigr) + (h_n/|x|)^{p'/2}\wedge 1,
\end{align}

\begin{comment}
\begin{align}\label{eq_bound_X_k_1} %\nonumber
&\P_{|k}\hspace*{-.01cm}\Big(\sup_{t \in \mathcal{T}_k^n}|\Delta X_{t}(k)| \geq |z|/2 \Big) \hspace*{-.05cm}\lesssim_{\,a.s.} \hspace*{-.05cm}\P_{|k}\hspace*{-.05cm}\left(\sup_{0 \leq t \leq 1}\hspace*{-.05cm}|W_t| \hspace*{-.05cm}\geq \hspace*{-.05cm}\frac{C_0\sigma_{k h_n} |x|}{\sigma_{k h_n} \hspace*{-.075cm}+ \hspace*{-.075cm}\sqrt{|x| \sigma_{k h_n}}}\hspace*{-.05cm} \right) \\& \nonumber \hspace*{.5cm}+ (h_n/|x|)^{p'/2}\wedge 1
\lesssim_{\,a.s.} \P\bigl(|W_1| \geq C_0 \sqrt{|x|}  \bigr) + (h_n/|x|)^{p'/2}\wedge 1,
\end{align}
\end{comment}
with some $C_0 > 0$. Next, observe that by the boundedness of $a_t$, it follows that
\begin{align}\label{eq_bound_exp_A_k_1}
A_k \lesssim_{\,a.s.} n \cd h_n^2 \max_{t \in [0,1]}|a_t| \lesssim_{\,a.s.} (\cd n)^{-1/3}.
\end{align}
We thus obtain for $y > 0$ the upper bound
\begin{align}\label{eq_bound_exp_T_k_1}
&\E_{|k}\Big[\exp\Big(-yT_k\Big)\Big] \lesssim_{\,a.s.} \P\bigl(|W_1| \geq C_0 \sqrt{|x|}\bigr) + (h_n/|x|)^{p'/2}\wedge 1\\
&\notag + \exp\Big(-C_1|x|yn h_n \sigma_{k h_n}/2 + C_2yn^{-1/3}\Big) + \1\bigl(x < 0\bigr),
\end{align}
where $p'$ arbitrarily large but finite. Note that elementary calculations yield that \eqref{eq_bound_exp_T_k_1} also supplies a bound for $\E_{|k}[\exp(-yV_k)]$. Next, observe that
\begin{align*}
U_k \leq n h_n \cd \sup_{kh_n \leq t \leq (k+1)h_n} \biggl| \int_{k h_n}^t (\sigma_s -\sigma_{kh_n}) \,d W_s \biggr| \stackrel{def}{=} U_k^+.
\end{align*}
Applying Burkholder's inequality, we get
\begin{align*}
\bigl\|_{|k}U_k^+\bigr\|_q \lesssim_{\,a.s.} n \cd h_n  \biggl\|_{|k}\int_{k h_n}^{(k+1)h_n} (\sigma_s -\sigma_{kh_n})^2\, ds \biggr\|_{q/2}^{1/2},
\end{align*}
hence another application of Burkholder's inequality yields that
\begin{align}\label{eq_U_K+}
\bigl\|_{|k}U_k^+\bigr\|_q \lesssim_{\,a.s.} n \cd  h_n^{3/2} h_n^{1/2} = (\cd n)^{-1/3} \mathcal{K}^{4/3},
\end{align}
where we also used that $\|\widetilde{a}\|_{\infty}, \|\widetilde{\sigma}\|_{\infty}, \|\widetilde{\eta}\|_{\infty} < \infty$ almost surely. Set $\mathcal{U}_{k,n} = \bigl\{U_k^+ \leq \delta \bigr\}$. By the Markov inequality and \eqref{eq_U_K+}, it follows for $q \geq 1$ and $\delta = (n \cd)^{-1/6}$
\begin{align}\label{eq_U_K+_and_set}
\P_{|k}\bigl(\mathcal{U}_{k,n}^c\bigr) \lesssim_{\,a.s.} \delta^{-q} \bigl(n \cd \bigr)^{-q/3} = \bigl(n \cd \bigr)^{-q/6}.
\end{align}
Using the power series of $\exp(x)$ and \eqref{eq_bound_exp_A_k_1}, we obtain by Cauchy-Schwarz
\begin{comment}
\begin{align*}
&\E_{|k}\Big[\exp\Big(-T_k\Big)\1_{\mathcal{U}_{k,n}}\Big] \\
&\leq \E_{|k}\Big[\exp\Big(-V_k\Big)\Big] + \E_{|k}\Big[\exp\Big(-V_k\Big)\bigl(\exp(U_k^+) - 1\bigr)\1_{\mathcal{U}_{k,n}}\Big]\\&\leq
\E_{|k}\Big[\exp\Big(-V_k\Big)\Big] + \bigl\|_{|k}\exp\Big(-V_k\Big)\bigr\|_2 \bigl\|_{|k}\bigl(\exp(U_k^+) - 1\bigr)\1_{\mathcal{U}_{k,n}}\bigr\|_2\\&\leq \E_{|k}\Big[\exp\Big(-V_k\Big)\Big] + C_3\bigl\|_{|k}\exp\Big(-V_k\Big)\bigr\|_2 \biggl(\sum_{k = 1}^K \frac{1}{k!}\E_{|k}\bigl[(U_k^+)^{2k}\bigr]^{1/2} + \delta^{K+1}\biggr),
\end{align*}
\end{comment}
\begin{align*}
&\E_{|k}\Big[\exp\Big(-T_k\Big)\1_{\mathcal{U}_{k,n}}\Big] \\
&\leq \E_{|k}\Big[\exp\Big(-V_k\Big)\Big] + \E_{|k}\Big[\exp\Big(-V_k\Big)\bigl(\exp(U_k^+ + A_k) - 1\bigr)\1_{\mathcal{U}_{k,n}}\Big]\\&\leq
\E_{|k}\Big[\exp\Big(-V_k\Big)\Big] + \bigl\|_{|k}\exp\Big(-V_k\Big)\bigr\|_2 \bigl\|_{|k}\bigl(\exp(U_k^+ + A_k) - 1\bigr)\1_{\mathcal{U}_{k,n}}\bigr\|_2\\&\leq \E_{|k}\Big[\exp\Big(-V_k\Big)\Big] + C_3\bigl\|_{|k}\exp\Big(-V_k\Big)\bigr\|_2 \biggl(\bigl\|_{|k} U_k^+ \bigr\|_2 + \delta^{2} + (\cd n)^{-1/3}\biggr),
\end{align*}
with some constant $C_3$. From \eqref{eq_U_K+}, this is bounded by
\begin{align}
\E_{|k}\Big[\exp\Big(-V_k\Big)\Big] + C_4\bigl\|_{|k}\exp\Big(-V_k\Big)\bigr\|_2 (\cd n)^{-1/3}, \quad C_4 > 0.
\end{align}
On the other hand, it follows from Cauchy-Schwarz that
\begin{align}
&\E_{|k}\Big[\exp\Big(-T_k\Big)\1_{\mathcal{U}_{k,n}^c}\Big] \leq \bigl\|_{|k} \exp\Big(-2T_k\Big) \bigr\|_2 \sqrt{\P_{|k}(\mathcal{U}_{k,n}^c)}.
\end{align}
Combining the above, we thus conclude from \eqref{eq_bound_exp_T_k_1} ($y = 2$, $p'$ large enough) and \eqref{eq_U_K+_and_set} ($q$ large enough) for some constant $C_5$:
\begin{align*}
&\E_{|k}\Big[\exp\Big(\hspace*{-.025cm}-\hspace*{-.025cm}T_k\Big)\Big] \leq_{a.s.}  \E_{|k}\Big[\exp\Big(\hspace*{-.025cm}-\hspace*{-.025cm}V_k\Big)\Big] \hspace*{-.05cm} + \hspace*{-.05cm}C_5\bigl\|_{|k}\exp\Big(\hspace*{-.025cm}-\hspace*{-.025cm}V_k\Big)\bigr\|_2 (\cd n)^{-1/3} \\&+ \hspace*{-.05cm}\bigl\|_{|k} \exp\Big(-2T_k\Big) \bigr\|_2 \hspace*{-.1cm} \sqrt{\P_{|k}(\mathcal{U}_{k,n}^c)} \hspace*{-.05cm} \leq_{a.s.} \E_{|k}\Big[\exp\Big(\hspace*{-.025cm}-\hspace*{-.025cm}V_k\Big)\Big]\hspace*{-.05cm} + \frac{C_5}{(\cd n)^{1/3}} \\&\times\hspace*{-.05cm}\biggl(\hspace*{-.05cm}\P\bigl(|W_1| \geq C_5 \sqrt{|x|}  \bigr) + \left(\frac{h_n}{|x|}\right)^{p'/4-1}\hspace*{-.5cm} + \1(x \leq 0) + \hspace*{-.05cm}\exp\big(\hspace*{-.05cm}-\hspace*{-.05cm}|x| n h_n \sigma_{k h_n}C_5 \big)\hspace*{-.05cm}\biggr)^{1/2}\,.
\end{align*}
In the same manner one obtains a lower bound and hence the claim follows (for $x \geq 0$), since $\inf_{0 \leq t \leq 1}\sigma_t \geq \sigma_- > 0$ almost surely.\\
{\bf Case $x < -1$:} Let
\begin{align}\nonumber
&\mathcal{V}_{n,k}(z) = \bigl\{\sup_{t \in \mathcal{T}_k^n}|\sigma_{k h_n}(W_{t} - W_{kh_n})| < |z|/4 \bigr\},\\
&\mathcal{X}_{n,k}(z) = \bigl\{\sup_{t \in \mathcal{T}_k^n}|\Delta X_{t}(k)| < |z|/2\bigr\},
\end{align}
and denote with $\mathcal{V}_{n,k}^c(z)$ and $\mathcal{X}_{n,k}^c(z)$ their complements. Observe that since
$|x| \geq 1$ and $\inf_{0 \leq s \leq 1}\sigma_s \geq \sigma_-$ and
\begin{align*}
\bigl|\Delta A_{t}(k)\bigr| \lesssim_{\,a.s.} h_n \max_{t \in [0,1]}|a_t| \lesssim_{\,a.s.} h_n,
\end{align*}
we get $0 = T_k = V_k$ for large enough $n$ on the set $\mathcal{V}_{n,k}(z) \cap \mathcal{X}_{n,k}(z)$. Hence we obtain from \eqref{eq_bound_X_k_1} that
\begin{align*}
&\bigl|\E_{|k}\Big[\Big(\exp(\hspace*{-.025cm}-\hspace*{-.025cm}T_k) - \exp(\hspace*{-.025cm}-\hspace*{-.025cm}V_k)\Big)\1_{\mathcal{V}_{n,k}(z)} \Big]\bigr| \leq \bigl|\E_{|k}\Big[\1_{\mathcal{V}_{n,k}(z) \cap \mathcal{X}_{n,k}^c(z)} \Big]\bigr| \\&\lesssim_{\,a.s.}\P\bigl(|W_1| \geq C_6 \sqrt{|x|}  \bigr) + (h_n/|x|)^{p'/2}\wedge 1,
\end{align*}
with some $C_6 > 0$. On the other hand, using very similar arguments as for the case $x \geq -1$, one derives that
\begin{align*}
&\bigl|\E_{|k}\Big[\Big(\exp(\hspace*{-.025cm}-\hspace*{-.025cm}T_k) - \exp(\hspace*{-.025cm}-\hspace*{-.025cm}V_k)\Big)\1_{\mathcal{V}_{n,k}^c(z)} \Big]\bigr|\\&\lesssim_{\,a.s.}\frac{1}{(\cd n)^{1/3}} \biggl(\hspace*{-.05cm}\P\bigl(|W_1| \geq C_6 \sqrt{|x|}  \bigr) + \left(\frac{h_n}{|x|}\right)^{p'/4-1}\hspace*{-.4cm} + \hspace*{-.05cm}\exp\big(\hspace*{-.05cm}-\hspace*{-.05cm}|x| n h_n \sigma_{k h_n}C_6 \big)\hspace*{-.05cm}\biggr)^{1/2},
\end{align*}
which completes the proof.
\end{proof}

Denote with $\mathcal{R}_{n,k}(\sigma)$ the version of $\mathcal{R}_{n,k}$ where $\sigma_t^2 = \sigma^2$ is constant for $t\in \big(\mathcal{T}_k^n\cup\mathcal{T}_{k-1}^n\big)$, and we use the same notation $\mathcal{L}_{n,k}(\sigma)$ for $\mathcal{L}_{n,k}$. %\begin{rem}\label{rem_replace}
It is apparent from the proof of Proposition \ref{prop_min_distrib_det_PPP2} that $\sigma_{k h_n}$ can be replaced with any $\sigma_{(k-j) h_n}$ where $j \geq 0$ is finite and independent of $n$.
%\end{rem}

\begin{lem}\label{lem_R_L_moments}
For $p \geq 1$ we have
\begin{align*}
\bigl\|h_n^{-1/2}\mathcal{R}_{n,k}\bigr\|_p, \, \bigl\|h_n^{-1/2}\mathcal{L}_{n,k}\bigr\|_p < \infty.
\end{align*}
\end{lem}

\begin{proof}[Proof of Lemma \ref{lem_R_L_moments}]
Using the identity
\begin{align*}
\E\bigl[|X|^p\bigr] = p \int_0^{\infty} x^{p-1}\bigl(\P(X > x )  + 1 - \P(X \geq -x)\bigr)dx,
\end{align*}
the claim follows from Proposition \ref{prop_min_distrib_det_PPP2} and the tower property of conditional expectation.
\end{proof}

\begin{lem}\label{lem_R_L_approx}
We have the equality
\begin{align*}
\E\bigl[h_n^{-1}\bigl(\mathcal{L}_{n,k} - \mathcal{R}_{n,k} \bigr)^2\bigl| \F_{(k-1)h_n}\bigr] = \Psi(\sigma_{(k-1)h_n}^2) + \OO_{a.s.}\bigl(h_n^{1/2}\bigr) \,.
\end{align*}
\end{lem}

\begin{proof}[Proof of Lemma \ref{lem_R_L_approx}]
%Denote with $k^* = k h_n$ and $k_{-}^* = (k-1)h_n$.
Using Proposition \ref{prop_min_distrib_det_PPP2} and relations \eqref{eq_exp_gen_1}, \eqref{eq_exp_gen_2}, one readily computes that
\begin{align}\nonumber\label{eq_L_mom_approx}
\E\bigl[\mathcal{L}_{n,k}\,|\,\F_{(k-1)h_n}\bigr] &= \E\bigl[\mathcal{L}_{n,k}(\sigma_{(k-1)h_n})\,|\,\F_{(k-1)h_n}\bigr] + \OO_{a.s.}\bigl(h_n\bigr),\\
\E\bigl[\mathcal{L}_{n,k}^2\,|\,\F_{(k-1)h_n}\bigr] &= \E\bigl[\mathcal{L}_{n,k}^2(\sigma_{(k-1)h_n})\,|\,\F_{(k-1)h_n}\bigr] + \OO_{a.s.}\bigl(h_n^{3/2}\bigr),
\end{align}
and similarly
\begin{align}\nonumber\label{eq_R_mom_approx}
\E\bigl[\mathcal{R}_{n,k}\,|\,\F_{(k-1)h_n}\bigr] &= \E\bigl[\mathcal{R}_{n,k}(\sigma_{(k-1)h_n})\,|\,\F_{(k-1)h_n}\bigr] + \OO_{a.s.}\bigl(h_n\bigr),\\
\E\bigl[\mathcal{R}_{n,k}^2\,|\,\F_{(k-1)h_n}\bigr] &= \E\bigl[\mathcal{R}_{n,k}^2(\sigma_{(k-1)h_n})\,|\,\F_{(k-1)h_n}\bigr] + \OO_{a.s.}\bigl(h_n^{3/2}\bigr).
\end{align}
Hence by \eqref{eq_L_mom_approx} and \eqref{eq_R_mom_approx}, we obtain
\begin{align}\label{eq_lem_R_L_approx_2}\nonumber
\E_{|k-1}\bigl[(\mathcal{L}_{n,k} - \mathcal{R}_{n,k})^2\bigr] &= \E_{|k-1}\bigl[\mathcal{L}_{n,k}^2(\sigma_{(k-1)h_n})\bigr] + \E_{|k-1}\bigl[\mathcal{R}_{n,k}^2(\sigma_{(k-1)h_n})\bigr]\\& \quad + \OO_{a.s.}\bigl(h_n^{3/2}\bigr) -2\E_{|k-1}\bigl[\mathcal{L}_{n,k}\mathcal{R}_{n,k}\bigr].
\end{align}
It thus suffices to consider the cross term in the last line. Note that by the tower property of conditional expectations
\begin{align*}
&\E_{|k-1}\bigl[\mathcal{L}_{n,k}\mathcal{R}_{n,k}\bigr] = \E_{|k-1}\bigl[\mathcal{L}_{n,k}\E_{|k}[\mathcal{R}_{n,k}]\bigr]\\&= \E_{|k-1}\bigl[\mathcal{L}_{n,k}\bigl(\E_{|k}[\mathcal{R}_{n,k}] - \E_{|k-1}[\mathcal{R}_{n,k}(\sigma_{(k-1)h_n})] + \E_{|k-1}[\mathcal{R}_{n,k}(\sigma_{(k-1)h_n})]\bigr) \bigr]\\&=\E_{|k-1}\bigl[\mathcal{L}_{n,k}\bigl(\E_{|k}[\mathcal{R}_{n,k}] - \E_{|k-1}[\mathcal{R}_{n,k}(\sigma_{(k-1)h_n})]\bigr)\bigr] \\
&\hspace*{4cm} +\E_{|k-1}\bigl[\mathcal{L}_{n,k}\E_{|k-1}[\mathcal{R}_{n,k}(\sigma_{(k-1)h_n})]\bigr].
\end{align*}
By Lemma \ref{lem_R_L_moments} and \eqref{eq_R_mom_approx}
\begin{align}\nonumber\label{eq_lem_R_L_approx_4}
\bigl|\E_{|k-1}\bigl[\mathcal{L}_{n,k}\bigl(\E_{|k}[\mathcal{R}_{n,k}] - \E_{|k-1}[\mathcal{R}_{n,k}(\sigma_{(k-1)h_n})]\bigr)\bigr]\bigr| &= \OO_{a.s.}\bigl(h_n\bigr) \E_{|k}\bigl[|\mathcal{L}_{n,k}|\bigr]\\ &= \OO_{a.s.}\bigl(h_n^{3/2}\bigr).
\end{align}
\begin{comment}
Next, observe that
\begin{align*}
&\E_{|k-1}\bigl[\mathcal{L}_{n,k}\E_{|k-1}[\mathcal{R}_{n,k}(\sigma_{(k-1)h_n})]\bigr] = \E_{|k-1}\bigl[\mathcal{L}_{n,k}\bigr] \E_{|k-1}\bigl[\mathcal{R}_{n,k}(\sigma_{(k-1)h_n})\bigr]\\&  = \bigl(\E_{|k-1}\bigl[\mathcal{L}_{n,k}\bigr] - \E_{|k-1}\bigl[\mathcal{L}_{n,k}(\sigma_{(k-1)h_n})\bigr]\bigr)\E_{|k-1}\bigl[\mathcal{R}_{n,k}(\sigma_{(k-1)h_n})\bigr] \\& \quad+ \E_{|k-1}\bigl[\mathcal{L}_{n,k}(\sigma_{(k-1)h_n})\bigr] \E_{|k-1}\bigl[\mathcal{R}_{n,k}(\sigma_{(k-1)h_n})\bigr].
\end{align*}
\end{comment}
The same arguments as above lead to
\begin{align}\nonumber\label{eq_lem_R_L_approx_5}
&\E_{|k-1}\bigl[\mathcal{L}_{n,k}\E_{|k-1}[\mathcal{R}_{n,k}(\sigma_{(k-1)h_n})]\bigr] \\&= \E_{|k-1}\bigl[\mathcal{L}_{n,k}(\sigma_{(k-1)h_n})\bigr] \E_{|k-1}\bigl[\mathcal{R}_{n,k}(\sigma_{(k-1)h_n})\bigr] + \OO_{a.s.}\bigl(h_n^{3/2}\bigr).
\end{align}
Since $\E_{|k-1}[\mathcal{R}_{n,k}(\sigma_{(k-1)h_n})]= \E_{|k}[\mathcal{R}_{n,k}(\sigma_{(k-1)h_n})]$,
%brownian motion/markov property
we have by the tower property of conditional expectations
\begin{align}\nonumber\label{eq_lem_R_L_approx_6}
&\E_{|k-1}\bigl[\mathcal{L}_{n,k}(\sigma_{(k-1)h_n})\bigr] \E_{|k-1}\bigl[\mathcal{R}_{n,k}(\sigma_{(k-1)h_n})\bigr] \\ \nonumber &= \E_{|k-1}\bigl[\E_{|k}[\mathcal{L}_{n,k}(\sigma_{(k-1)h_n}) \mathcal{R}_{n,k}(\sigma_{(k-1)h_n})]\bigr]\\&= \E_{|k-1}\bigl[\mathcal{L}_{n,k}(\sigma_{(k-1)h_n})\mathcal{R}_{n,k}(\sigma_{(k-1)h_n})\bigr].
\end{align}
Combining \eqref{eq_lem_R_L_approx_4}, \eqref{eq_lem_R_L_approx_5} and \eqref{eq_lem_R_L_approx_6}, we obtain
\begin{align*}
\E_{|k-1}\bigl[\mathcal{L}_{n,k}\mathcal{R}_{n,k}\bigr] = \E_{|k-1}\bigl[\mathcal{L}_{n,k}(\sigma_{(k-1)h_n})\mathcal{R}_{n,k}(\sigma_{(k-1)h_n})\bigr] + \OO_{a.s.}\bigl(h_n^{3/2}\bigr),
\end{align*}
and hence the claim follows.
\end{proof}

\begin{lem}\label{lem_tech_prelim_thm}
Let $\tilde{\bf \Psi}(x) = \Psi(x^2)$, $a_{n,l} = l r_n^{-1}/2$, $s_{n,l} = l h_n r_n^{-1} = 2a_{n,l} h_n$ and
$
M_{k,n} = \bigl(m_{n,2k} - m_{n,2k-1}\bigr)^2 2 h_n^{-1} r_n.
$
We then have the following upper bounds
\begin{align*}
&\text{{\bf (i)}}\quad\biggl\|\sum_{l = 1}^{h_n^{-1} r_n} \bigl(\Psi^{-1}\bigr)'\bigl(\Psi(\sigma_{s_{n,l-1}}^2)\bigr)r_n\sum_{k = a_{n,l-1}+1}^{a_{n,l}}\bigl(\Psi(\sigma_{kh_n}^2) - \Psi(\sigma_{s_{n,l-1}}^2) \bigr)\biggr\|_1 = \OO\bigl(1\bigr),\\
&\text{{\bf (ii)}} \quad \sum_{l = 1}^{h_n^{-1} r_n} r_n^2\biggl\|\sum_{k = a_{n,l-1}+1}^{a_{n,l}}\bigl(\Psi(\sigma_{kh_n}^2) - \Psi(\sigma_{s_{n,l-1}}^2)\bigr)\biggr\|_2^2 = \OO\bigl(1\bigr),\\
&\text{{\bf (iii)}}\quad\biggl\|\sum_{l=1}^{r_nh_n^{-1}} \sigma_{s_{n,l-1}}^2h_nr_n^{-1} - \sum_{l=1}^{r_nh_n^{-1}}\sum_{k = a_{n,l-1}+1}^{a_{n,l}} \sigma_{k h_n}^2h_nr_n^{-1}\biggr\|_1\lesssim n^{-1/3}.
\end{align*}
\end{lem}

\begin{proof}[Proof of Lemma \ref{lem_tech_prelim_thm}]
Note first that due to \eqref{eq_thm_psi_prop_3} and \eqref{eq_inverse_dif_prop} below, it follows that for any $p \geq 1$
\begin{align}\label{eq_upsilon_uniformly_bounded}
\sup_{t \in [0,1]}\bigl\|\tilde{\bf \Psi}'(\sigma_t)\bigr\|_p, \, \sup_{x \geq \sigma_-}|\tilde{\bf \Psi}''(x)| < \infty.
\end{align}
{\bf Case} (i): By It\^{o}'s formula and the independence of $W_s$ and $W_s^{\bot}$ in Assumption \ref{sigma}
\begin{align}
\tilde{\bf \Psi}(\sigma_t) - \tilde{\bf \Psi}(\sigma_s) = \int_s^t \tilde{\bf \Psi}'(\sigma_r) \,d \sigma_r + \frac{1}{2}\int_s^t \tilde{\bf \Psi}''(\sigma_r) \bigl(\tilde{\sigma}_r^2 + \tilde{\eta}_r^2\bigr)\,dr.
\end{align}
Using \eqref{eq_upsilon_uniformly_bounded}, we thus obtain
\begin{align*}
&\biggl\|\sum_{l = 1}^{h_n^{-1} r_n} \bigl(\Psi^{-1}\bigr)'\bigl(\Psi(\sigma_{s_{n,l-1}}^2)\bigr)r_n\sum_{k = a_{n,l-1}+1}^{a_{n,l}}\bigl(\tilde{\bf \Psi}(\sigma_{kh_n}) - \tilde{\bf \Psi}(\sigma_{s_{n,l-1}}) \bigr)\biggr\|_1 \\&\lesssim \biggl\|\sum_{l = 1}^{h_n^{-1} r_n} \bigl(\Psi^{-1}\bigr)'\bigl(\Psi(\sigma_{s_{n,l-1}}^2)\bigr)r_n\sum_{k = a_{n,l-1}+1}^{a_{n,l}}\int_{a_{n,l} h_n}^{k h_n} \tilde{\bf \Psi}'(\sigma_r)\bigl(\tilde{\sigma}_r \, d W_r + \tilde{\eta}_r\, d W_r^{\bot} \bigr)\biggr\|_1 \\&\quad + \sum_{l = 1}^{h_n^{-1} r_n} r_n \sum_{k = a_{n,l-1}+1}^{a_{n,l}} (a_{n,l} - k) h_n\\& \lesssim \biggl\|\sum_{l = 1}^{h_n^{-1} r_n} \bigl(\Psi^{-1}\bigr)'\bigl(\Psi(\sigma_{s_{n,l-1}}^2)\bigr)r_n\sum_{k = a_{n,l-1}+1}^{a_{n,l}}(a_{n,l} - k)\\
&\hspace*{4cm}\times \int_{kh_n}^{(k+1) h_n}\tilde{\bf \Psi}'(\sigma_r)\bigl(\tilde{\sigma}_r \, d W_r + \tilde{\eta}_r \,d W_r^{\bot} \bigr)\biggr\|_1 + \OO\bigl(1\bigr).
\end{align*}
Jensen's inequality gives a bound with the $\|\,\cdot\,\|_2$ norm for the squared  $\|\,\cdot\,\|_1$ norm above and Burkholder's inequality and \eqref{eq_upsilon_uniformly_bounded} then yield the upper bound \(\sum_{l = 1}^{h_n^{-1} r_n} r_n^2 \sum_{k = a_{n,l-1}+1}^{a_{n,l}}(a_{n,l} - k)^2 h_n = \OO\bigl(1\bigr)\).
%{\color{red}{Delete:
%\begin{align*}
%&\biggl\|\sum_{l = 1}^{h_n^{-1} r_n} \bigl(\Psi^{-1}\bigr)'\bigl(\Psi(\sigma_{s_{n,l-1}}^2)\bigr)r_n\sum_{k = a_{n,l-1}+1}^{a_{n,l}}(a_{n,l} - k)\int_{kh_n}^{(k+1) h_n}\tilde{\bf \Psi}'(\sigma_r)\bigl(\tilde{\sigma}_r \,dW_r + \tilde{\eta}_r \,d W_r^{\bot} \bigr)\biggr\|_1^2 \\&\lesssim \sum_{l = 1}^{h_n^{-1} r_n} r_n^2 \sum_{k = a_{n,l-1}+1}^{a_{n,l}}(a_{n,l} - k)^2 h_n = \OO\bigl(1\bigr).
%\end{align*}}}
Combining all bounds, we thus obtain (i). Case (ii) and (iii) can be handled in a very similar manner and we omit the proofs.
\end{proof}

\begin{proof}[Proof of Theorem \ref{theoupper} and Corollary \ref{corrupper}]
Observe that ${\bar M}_{k,n} = M_{k,n} - \E_{}\bigl[M_{k,n}|\F_{(k-1)h_n}\bigr]$ is a sequence of martingale differences. Lemma \ref{lem_R_L_moments} yields that all moments of ${\bar M}_{k,n}$ exist. Hence for any index set $\mathcal{J} \subseteq \bigl\{0,\ldots,h_n^{-1}-1\bigr\}$, Burkholder's inequality ensures that for any $p \geq 1$
\begin{align}\label{eq_rosenthal_M}
\biggl\|\sum_{k \in \mathcal{J}} {\bar M}_{k,n} \biggr\|_p \lesssim r_n \sqrt{|\mathcal{J}|},
\end{align}
where $|\mathcal{J}|$ is the cardinality of the set $\mathcal{J}$. Let
\begin{align*}
\mathcal{M}_l = \biggl\{\sum_{k = a_{n,l-1}+1}^{a_{n,l}} \bar{M}_{k,n} +\Psi\bigl(\sigma_{s_{n,l-1}}^2 \bigr)/2 > 0 \biggr\}, \quad \text{and $\mathcal{M} = \bigcap_{l = 1}^{h_n^{-1}-1} \mathcal{M}_l$.}
\end{align*}
Proposition \ref{thm_psi_properties} yields that $\Psi\bigl(x^2\bigr) > 0$ for $x  > 0$. %ABSICHTLICH x und nicht \sigma!!!!!
Then we obtain from the Markov inequality and \eqref{eq_rosenthal_M} that
\begin{align}\nonumber \label{eq_union_sets_M_complement_negligible}
\P\biggl(\bigcup_{l = 0}^{h_n^{-1}-1} \mathcal{M}_l^c \biggr) &\leq \sum_{l = 0}^{h_n^{-1} - 1} \P\bigl(\mathcal{M}_l^c \bigr) \lesssim 2^p\sum_{l = 0}^{h_n^{-1} - 1}  \Psi\bigl(\sigma_{s_{n,l-1}}^2 \bigr)^{-p}\biggl\|\sum_{k = a_{n,l-1}+1}^{a_{n,l}}{\bar M}_{k,n}\biggr\|_p^p \\&\lesssim
\sum_{l = 0}^{h_n^{-1} - 1} r_n^{p/2} = \oo\bigl(1\bigr),
\end{align}
for $p >4$. We are now ready to proceed to the main proof. From \eqref{estimatorPPP} it follows that
\begin{align*}
&\widetilde{IV}_n^{h_n,r_n}-\int_0^1\sigma_t^2\,dt=\sum_{l=1}^{r_nh_n^{-1}}
\sigma_{s_{n,l-1}}^2h_nr_n^{-1}-\int_0^1\sigma_t^2\,dt\\
&\quad +\hspace*{-0.05cm}\sum_{l=1}^{r_nh_n^{-1}}\hspace*{-0.1cm}\Bigg(\Psi^{-1}\hspace*{-0.015cm}
\Bigg(\sum_{k=a_{n,l-1}+1}^{a_{n,l}}\hspace*{-0.15cm}(m_{n,2k}-m_{n,2k-1})^22\,h_n^{-1}r_n
\hspace*{-0.05cm}\Bigg)\hspace*{-0.05cm}-\sigma_{s_{n,l-1}}^2\hspace*{-0.05cm}\Bigg)
\hspace*{-0.015cm}h_nr_n^{-1}.
\end{align*}
Consider first the approximation error in the quadratic variation by setting the volatility locally constant on the blocks of the coarse grid. Due to Lemma \ref{lem_tech_prelim_thm} (iii), it suffices to consider the error
\begin{align*}
\biggl\|\sum_{l = 1}^{h_n^{-1}} \sigma_{lh_n}^2 h_n  -\int_0^1\sigma_t^2\,dt\biggr\|_2 \leq \sum_{l = 1}^{h_n^{-1}} \int_{(l-1)h_n}^{lh_n}\bigl\|\sigma_t^2 - \sigma_{lh_n}^2\big\|_2\,dt,
\end{align*}
which by the triangle and Burkholder's inequality is bounded by
\begin{align*}
\lesssim \sum_{l = 1}^{h_n^{-1}} \int_{(l-1)h_n}^{lh_n} \bigl(\sqrt{h_n} + h_n \|\tilde a\|_{\infty}\bigr)\,dt \lesssim n^{-1/3}.
\end{align*}
In order to bound the remaining estimation error
\begin{align*}
\sum_{l=1}^{r_nh_n^{-1}}\hspace*{-0.05cm}\Bigg(\Psi^{-1}\Bigg(\sum_{k = a_{n,l-1}+1}^{a_{n,l}} M_{k,n}\Bigg)-\sigma_{s_{n,l-1}}^2\Bigg)
h_nr_n^{-1},
\end{align*}
we use a Taylor expansion and that the first two derivatives of $\Psi^{-1}$ exist and are bounded according to Proposition \ref{thm_psi_properties} below. To this end, denote with
\begin{align*}
\Delta_{k,l,n}(M,\Psi) = \E_{}\bigl[M_{k,n}|\F_{(k-1)h_n}\bigr] - \Psi(\sigma_{kh_n}^2) + \Psi(\sigma_{kh_n}^2) - \Psi(\sigma_{s_{n,l-1}}^2).
\end{align*}
It then follows that
\begin{align*}
& \Psi^{-1}\biggl(\sum_{k = a_{n,l-1}+1}^{a_{n,l}} M_{k,n} -  \Psi\bigl(\sigma_{s_{n,l-1}}^2 \bigr) +  \Psi\bigl(\sigma_{s_{n,l-1}}^2 \bigr) \biggr) \\&= \sigma_{s_{n,l-1}}^2 + \bigl( \Psi^{-1}\bigr)'\bigl( \Psi(\sigma_{s_{n,l-1}}^2)\bigr)\biggl(\sum_{k = a_{n,l-1}+1}^{a_{n,l}} \bigl(\bar{M}_{k,n} + \Delta_{k,l,n}(M,\Psi)\bigr)\biggr) \\&\hspace*{1.575cm}+ \frac{1}{2}\bigl( \Psi^{-1}\bigr)''\bigl( \xi_l\bigr)\biggl(\sum_{k = a_{n,l-1}+1}^{a_{n,l}}  \bigl(\bar{M}_{k,n} + \Delta_{k,l,n}(M,\Psi)\bigr)\biggr)^2 \\&\stackrel{def}{=} \sigma_{s_{n,l-1}}^2 + \Delta\bigl( \Psi\bigr)_{l,1} + \Delta\bigl( \Psi\bigr)_{l,2},
\end{align*}
where $\xi_l \geq  \Psi\bigl(\sigma_{s_{n,l-1}}^2 \bigr)/2 > 0$  on the set $\mathcal{M}_l$. We first deal with $\Delta\bigl( \Psi\bigr)_{l,1}$. To this end, denote with $Z_l = \sum_{k = a_{n,l-1}+1}^{a_{n,l}} \bar{M}_{k,n}$, which is a partial sum of martingale differences. Hence, by Burkholder's inequality (see also \eqref{eq_rosenthal_M}), we obtain
\begin{align}\label{eq_DELTA_psi_1.1}
\biggl\|\sum_{l = 1}^{r_nh_n^{-1}} \bigl( \Psi^{-1}\bigr)'\bigl( \Psi(\sigma_{s_{n,l-1}}^2)\bigr) Z_l \biggr\|_2^2 \lesssim \sum_{l = 1}^{r_nh_n^{-1}} \bigl\|Z_l\bigr\|_2^2 \lesssim r_n^2 h_n^{-1} = \OO\bigl(1\bigr).
\end{align}
On the other hand, we obtain from Lemma \ref{lem_R_L_approx} and Lemma \ref{lem_tech_prelim_thm} (i) that
\begin{align}\label{eq_DELTA_psi_1.2}
\biggl|\sum_{l = 1}^{r_nh_n^{-1}} \bigl( \Psi^{-1}\bigr)'\bigl( \Psi(\sigma_{s_{n,l-1}}^2)\bigr)\sum_{k = a_{n,l-1}+1}^{a_{n,l}} \Delta_{k,l,n}(M,\Psi) \biggr| = \OO_{\P}\bigl(1\bigr).
\end{align}
Combining \eqref{eq_DELTA_psi_1.1} and \eqref{eq_DELTA_psi_1.2}, we find
\begin{align}\label{eq_DELTA_psi_1.3}
\biggl|\sum_{l = 1}^{r_nh_n^{-1}}  \Delta\bigl( \Psi\bigr)_{l,1} \1_{\mathcal{M}}\biggr| = \OO_{\P}\bigl(1\bigr).
\end{align}
In the same manner, but using Lemma \ref{lem_tech_prelim_thm} (ii) and additionally $\|( \Psi^{-1})''( \xi_l)\1_{\mathcal{M}_l}\|_\infty<\infty$ by Proposition \ref{thm_psi_properties} below, we obtain
\begin{align*}
\biggl\|\sum_{l = 1}^{r_nh_n^{-1}} \Delta\bigl( \Psi\bigr)_{l,2}\1_{\mathcal{M}} \biggr\|_1 &\lesssim \sum_{l = 1}^{r_nh_n^{-1}}\Big( \bigl\|Z_l\bigr\|_2^2+\big\|\sum_{k = a_{n,l-1}+1}^{a_{n,l}} \Delta_{k,l,n}(M,\Psi)\big\|_2^2\Big)=\mathcal{O}(1).
\end{align*}
Since $\P\bigl(\mathcal{M}^c\bigr) = \oo\bigl(1\bigr)$ by \eqref{eq_union_sets_M_complement_negligible}, this suffices to guarantee that
\begin{align*}
\widetilde{IV}_n^{h_n,r_n}-\int_0^1\sigma_t^2\,dt=\mathcal{O}_{\P}\big(n^{-1/3}\big)\,.%\\  \noalign{\hfill\qed}
\end{align*}
%\vspace*{-1cm}
Based on a Taylor expansion for $\Psi_n^{-1}$ and using analogous bounds and Proposition \ref{prop_psi_emp}, we obtain likewise
\[\widehat{IV}_n^{h_n,r_n}-\int_0^1\sigma_t^2\,dt=\mathcal{O}_{\P}\big(n^{-1/3}\big)\,.\]
and conclude Corollary \ref{corrupper}. This completes the proof of Theorem \ref{theoupper} and Corollary \ref{corrupper} in absence of jumps for estimators \eqref{estimatorPPP} and \eqref{estimator}, respectively.

Finally, consider the truncated estimators. Since $\tau< \infty$ and $\P(\sup_{t \in [0,1]\setminus \mathcal{V}_n}\sigma_t^2 \leq \tau) = 1$, the arguments above reveal that it suffices to show
\begin{align*}%\label{eq_prop_violations_1}
\sum_{l=1}^{r_nh_n^{-1}}\hspace*{-0.05cm}\Bigg|\Bigg(\tau\wedge \Psi^{-1}\Big(\sum_{k = a_{n,l-1}+1}^{a_{n,l}} M_{k,n}\Big)\Bigg)-\sigma_{s_{n,l-1}}^2\Bigg|\1\Bigl(\tfrac{l h_n}{r_n} \in \mathcal{V}_n\Bigr)\tfrac{h_n}{r_n} = \OO_{\P}\bigl(n^{-1/3}\bigr),
\end{align*}
uniformly for $n \in \N$. However, since we have that $\sup_{0 \leq t \leq 1}\sigma_t^2 < \infty$ almost surely and
\(
\sum_{l=1}^{r_nh_n^{-1}}\1\bigl(l h_n/r_n \in \mathcal{V}_n\bigr) \) is finite almost surely,
the left-hand side above is bounded by
\begin{align*}
\sum_{l=1}^{r_nh_n^{-1}}\hspace*{-0.05cm}\OO_{\P}\bigl(\tau + \sup_{0 \leq t \leq 1}\sigma_t^2 \bigr) \1\bigl(l h_n/r_n \in \mathcal{V}_n\bigr)h_nr_n^{-1} &= \OO_{\P}\bigl(\tau+ \sup_{0 \leq t \leq 1}\sigma_t^2\bigr) |\mathcal{V}_n| n^{-1/3} \\& = \OO_{\P}\bigl(n^{-1/3}\bigr).
\end{align*}
Therefore, it suffices to show that $\int_0^1 \sigma_s^2 \1\bigl(s \in \mathcal{V}_n \bigr)\,d s = \OO_{\P}(n^{-1/3})$. From
\begin{align*}
\int_0^1 \sigma_s^2 \1\bigl(s \in \mathcal{V}_n \bigr)ds \leq \sup_{0 \leq s \leq 1} \sigma_s^2 |\mathcal{V}_n|  r_n^{-1} h_n = \OO_{\P}\bigl(n^{-1/3}\bigr)\,,
\end{align*}
the claim follows.
\end{proof}

\subsection{Properties of $ \Psi$\label{theoupperb}}
\label{proppsiapp}
It follows from Proposition \ref{prop_min_distrib_det_PPP2} and \eqref{eq_exp_gen_1}, \eqref{eq_exp_gen_2} that  for  $\tilde{\bf\Psi}(\sigma)=\Psi(\sigma^2)$
\begin{align}\label{eq_var_decomposition_psi}
h_n^{-1} \E\bigl[(m_{n,k} - m_{n,k-1})^2\bigr]& = \tilde{\bf \Psi}(\sigma) + \OO\bigl((\lambda n)^{-1/3}\bigr).
\end{align}
Having understood the behaviour of $\tilde{\bf \Psi}(\sigma)$, analogue properties of $\Psi(\sigma^2)$ readily follow. Let
\begin{align}H(x) = \int_0^1(W_t+x)_+\,dt.\end{align}
Then by \eqref{eq_exp_gen_1}, \eqref{eq_exp_gen_2}, we derive
\begin{align}\nonumber
\tilde{\bf \Psi}\bigl(\sigma\bigr) &=  \nonumber  4\sigma^2\int_0^{\infty}x\biggl(\E\left[e^{-\KK\sigma H(x)}\right] + 1 - \E\left[e^{-\KK\sigma H(-x)}\right]
\biggr)dx \\&\quad -2 \sigma^2 \biggl(\int_0^{\infty}\biggl(\E\left[e^{-\KK\sigma H(x)}\right] - 1 + \E\left[e^{-\KK\sigma H(-x)}\right]
\biggr)dx \biggr)^2\label{tildepsi}.
\end{align}
Next, consider the distribution on the negative half axis.
With $x<0$, we make the decomposition
\begin{align*}
&\E\biggl[e^{-\sigma \mathcal{K}\int_0^1 (W_t - x)_-dt} \biggr] \\&= \E\biggl[e^{-\sigma \mathcal{K}\int_0^1 (W_t - x)_-dt}\1\big(\inf_{0 \leq t \leq 1}W_t \leq x\big) + \1\big(\inf_{0 \leq t \leq 1}W_t \geq x\big) \biggr]\\&\stackrel{def}{=} U_1(x) + U_2(x).
\end{align*}
Let $T_x$ be the first passage time of $W$ to level $x$ with density
\begin{align*}
f_{T_x}(t) = \frac{|x|}{\sqrt{2 \pi t^3}} e^{-x^2/2t}\,,t\ge 0\,,
\end{align*}
see (6.3) in Section 2.6 of \cite{karatzas}. From $\bigl\{T_x \leq 1\bigr\} = \bigl\{\inf_{0 \leq t \leq 1}W_t \leq x\bigr\}$
it follows from the strong Markov property of $W$ that
\begin{align*}
U_1(x) &=% \int_0^{\infty} \E\biggl[e^{-\sigma \mathcal{K}\int_{s}^1 (W_t - x)_-dt}\biggl| T_x = s\biggr] f_{T_x}(s) ds \\
\int_0^{1} \E\biggl[e^{-\sigma \mathcal{K}\int_{s}^1 (W_t - x)_-dt}\biggl| T_x = s\biggr] f_{T_x}(s) ds\\&= \int_0^{1} \E\biggl[e^{-\sigma \mathcal{K}\int_{0}^{1-s} (W_t)_-dt}\biggr] f_{T_x}(s)ds.
\end{align*}
Using a time shift yields
\begin{align*}
U_1(x) &= \int_0^{1} \E\biggl[e^{-\sigma \mathcal{K}(1-s)^{3/2} \int_{0}^1 (W_t)_-dt}\biggr] f_{T_x}(s)ds.
\end{align*}
We then obtain that
\begin{align}\label{eq_distrib_x_lower_zero}
\nonumber\E\biggl[e^{-\sigma \mathcal{K} \int_0^1 (W_t - x)_-dt} \biggr] &= \P\bigl(\inf_{0 \leq t \leq 1}W_t \geq x\bigr) + \int_0^{1} \E\biggl[e^{-\sigma \mathcal{K}(1-s)^{3/2} \int_{0}^1 (W_t)_-dt}\biggr] f_{T_x}(s)ds\\&\stackrel{def}{=}\P\bigl(\inf_{0 \leq t \leq 1}W_t \geq x\bigr) + A^-(x), \quad \text{for $x < 0$.}
\end{align}
Let $I(\KK\sigma,s) = \E\biggl[e^{-\KK\sigma (1-s)^{3/2} \int_{0}^1 (W_t)_-dt}\biggr]$. Then by \eqref{eq_distrib_x_lower_zero}
\begin{align}\nonumber
&\int_0^{\infty}x\bigl(1 - \E\left[e^{-\KK\sigma H(-x)}\right]\bigr)dx =\nonumber \int_0^{\infty}x\bigl(\P\big(\inf_{0 \leq t \leq 1}W_t < - x\big) - A^-(-x)\bigr)dx \\ \nonumber &= \int_0^{\infty}x \P(\inf_{0 \leq t \leq 1}W_t < - x)dx - \int_0^1 I(\KK\sigma,s)\int_0^{\infty} x f_{T_x}(s)dx ds \\ \nonumber&= \frac{1}{2} - \frac{1}{2}\int_0^1 I(\KK\sigma,s)ds,
\end{align}
since \(\int_0^{\infty}x \P(\inf_{0 \leq t \leq 1}W_t < - x)dx= \frac{1}{2}\).
Likewise, it follows that
\begin{align}\nonumber
\int_0^{\infty}\bigl(1 - \E\left[e^{-\KK\sigma H(-x)}\right]\bigr)dx &= \int_0^{\infty}\P\big(\inf_{0 \leq t \leq 1}W_t < - x\big)dx - \sqrt{\frac{2}{\pi}}\int_0^1 I(\KK\sigma,s)ds\\& = \sqrt{\frac{2}{\pi}} - \sqrt{\frac{2}{\pi}}\int_0^1 I(\KK\sigma,s)ds.
\end{align}
We thus obtain
\begin{align}\nonumber \label{eq_defn_M}
\tilde{\bf \Psi}\bigl(\sigma\bigr) &= 4 \sigma^2\biggl(\int_0^{\infty}x\E\left[e^{-\KK\sigma H(x)}\right]dx + \frac{1}{2} - \frac{1}{2}\int_0^1 I(\KK\sigma,s)ds\biggr) \\ \nonumber &\quad - 2 \sigma^2\biggl(\int_0^{\infty}\E\left[e^{-\KK\sigma H(x)}\right]dx - \sqrt{\frac{2}{\pi}} + \sqrt{\frac{2}{\pi}}\int_0^1 I(\KK\sigma,s)ds\biggr)^2\\&=2\sigma^2\bigl(2 \Lambda_1(\sigma) - \Lambda_2^2(\sigma)\bigr)\,,
\end{align}
with functionals $\Lambda_1,\Lambda_2$. In the sequel, we write $\partial^k f(x) = \partial^k f(x)/\partial^k x$.
The further analysis of properties of $\tilde{\bf{\Psi}}$ is structured in several lemmas which combined imply the following key proposition.
%\MR proposition nur bis zweite Ableitung, ueberpruefe O-Terme, auch im Beweis!

\begin{prop}\label{thm_psi_properties}
Suppose that $\sigma \geq \sigma_0 > 0$, $\KK \geq C(\sigma_0)$ for $C(\sigma_0)$ sufficiently large (the exact value of $C(\sigma_0)$ follows from \eqref{eq_KK_condition}). Then we have uniformly for $\sigma \geq \sigma_0$
\begin{align}\nonumber \label{eq_thm_psi_prop_1}
\partial \tilde{\bf \Psi}\bigl(\sigma\bigr) &= 4\sigma \left(1 - \frac{2}{\pi}\right) + \OO\left(\frac{\sigma^{\frac{2}{3}}}{\KK^{\frac{1}{3}}}\right)  > 0\quad \text{and} \\ \tilde{\bf \Psi}\bigl(\sigma\bigr) &= 2\sigma^2\left(1 - \frac{2}{\pi}\right) + \OO\left(\frac{\sigma^{\frac{2}{3}}}{\KK^{\frac{1}{3}}}\right) > 0.
\end{align}
Moreover, it holds that
\begin{align}\label{eq_thm_psi_prop_3}
\sup_{\sigma \geq \sigma_0}\biggl|\partial^2 \tilde{\bf \Psi}(\sigma)\biggr| < \infty.
\end{align}
\end{prop}
Using the relation
\begin{align}\label{eq_inverse_dif_prop}
\partial\tilde{\bf \Psi}^{-1}(\varrho) =  \frac{1}{\partial\tilde{\bf \Psi}(\sigma)}, \quad \tilde{\bf \Psi}(\sigma) = \varrho,
\end{align}
we get that the second derivative is uniformly bounded for $\sigma \geq \sigma_0 = \tilde{\bf \Psi}^{-1}(\varrho_0)$, i.e.
\begin{align}\label{eq_bound_second_inverse_diff_Psi}
\sup_{\varrho \geq \varrho_0}\bigl|\partial^2 \tilde{\bf \Psi}^{-1}(\varrho)\bigr| = \sup_{\sigma \geq \sigma_0}\biggl|\frac{\partial^2\tilde{\bf \Psi}(\sigma)}{\bigl(\partial\tilde{\bf \Psi}(\sigma)\bigr)^3}\biggr| < \infty.
\end{align}
So far we have focused on results for $\tilde{\bf \Psi}(\sigma)=\Psi(\sigma^2)$. Essentially the same results are valid for $\Psi_n(\sigma^2)$, which we state now.

\begin{prop}\label{prop_psi_emp}
Introduce
\begin{align*}
B_{n,1} &= \int_0^{\infty} x \P\biggl(\max_{0 \leq i \leq n h_n-1} W_{i/(nh_n)} \geq x \biggr)dx   \quad \text{and}
\\B_{n,2} &= \int_0^{\infty} \P\biggl(\max_{0 \leq i \leq n h_n-1} W_{i/(nh_n)} \geq x \biggr)dx,
\end{align*}
which satisfy $B_{n,1} \to \frac{1}{2}$, $B_{n,2} \to \sqrt{\frac{2}{\pi}}$. Then \eqref{eq_thm_psi_prop_1} and \eqref{eq_thm_psi_prop_3} in Proposition \ref{thm_psi_properties} remain valid if we replace $\tilde{\bf \Psi}\bigl(\sigma\bigr)$ with $ \Psi_n\bigl(\sigma^2\bigr)$ and $1 - \frac{2}{\pi}$ with $2 B_{n,1} - B_{n,2}^2$. Likewise, \eqref{eq_bound_second_inverse_diff_Psi} also holds.
\end{prop}

\begin{proof}[Proof of Proposition \ref{thm_psi_properties}]
We write shortly $\Lambda_1 = \Lambda_1(\sigma)$, $\Lambda_2 = \Lambda_2(\sigma)$. We have that
\begin{align*}
\partial \tilde{\bf \Psi}\bigl(\sigma\bigr) &= 4 \sigma \bigl(2 \Lambda_1 - \Lambda_2^2\bigr) + 2 \sigma^2\left(2\partial \Lambda_1 - 2 \Lambda_2 \partial \Lambda_2\right).
\end{align*}
Using Lemmas \ref{lem_xp} and \ref{lem_bound_I_KK_plus_derivative} from below, we obtain with \eqref{eq_defn_M}
\begin{align*}
&\bigl|\Lambda_1 - \frac{1}{2}\bigr| \leq 6 (\KK \sigma)^{-\frac{2}{3}} + \frac{3}{2}\left(\frac{ \KK \sigma}{\log(\KK \sigma)}\right)^{-\frac{2}{5}} \stackrel{def}{=} R_{\Lambda_1}\\
&\bigl|\Lambda_2 - \sqrt{\frac{2}{\pi}}\bigr| \leq 2 \Bigl( \sqrt{\frac{2}{\pi}} + 1\Bigr)(\KK \sigma)^{-\frac{1}{3}} + 3\sqrt{\frac{2}{\pi}}\left(\frac{ \KK \sigma}{\log(\KK \sigma)}\right)^{-\frac{2}{5}} \stackrel{def}{=} R_{\Lambda_2}.
\end{align*}
Moreover, applying Lemmas \ref{lem_HK}, \ref{lem_HK_x} and \ref{lem_bound_I_KK_plus_derivative} yields
\begin{align*}
&\bigl|\partial \Lambda_1 \bigr| \leq 6 \bigl( \KK \sigma^3\bigr)^{-\frac{1}{2}} + \frac{3}{2 \sigma} \left(\frac{\KK \sigma}{\log(\KK \sigma)}\right)^{-\frac{2}{5}} \stackrel{def}{=} R_{\partial \Lambda_1}, \quad \Lambda_2^2 \leq \frac{2}{\pi},\\
&\bigl|\Lambda_2 \partial \Lambda_2 \bigr| \leq \biggl(R_{\Lambda_2} + \sqrt{\frac{2}{\pi}}\biggr) \biggl(4\bigl(1 + (2\pi)^{-\frac12}\bigr)\KK^{-\frac{1}{3}} \sigma^{-\frac{4}{3}} + \sqrt{\frac{2}{\pi}} \frac{3}{\sigma} \left(\frac{ \KK \sigma}{\log(\KK \sigma)}\right)^{-\frac{2}{5}} \biggr)\\&  \qquad \quad\, \stackrel{def}{=} R_{\partial \Lambda_2}.
\end{align*}
We thus obtain from the above that
\begin{align*}
\bigl|\partial \tilde{\bf \Psi} - 4 \sigma\bigl(1 - \frac{2}{\pi}\bigr)\bigr|&\leq 4 \sigma\biggl(2 R_{\Lambda_1} + R_{\Lambda_2} \Big(R_{\Lambda_2} + 2 \sqrt{\frac{2}{\pi}}\Big)\biggr) \\& \quad + 4 \sigma^2 \bigl(R_{\partial \Lambda_1} + R_{\partial \Lambda_2} \bigr) =\mathcal{O}\Big( \KK^{-\frac{1}{3}} \sigma^{\frac{2}{3}}\Big),
\end{align*}
\begin{align*}
\bigl|\tilde{\bf \Psi} - 2 \sigma^2 \Bigl(1 - \sqrt{\frac{2}{\pi}}\Bigr) \bigr| &\leq 2 \sigma^2 \biggl(R_{\Lambda_1} + R_{\Lambda_2} \Big(R_{\Lambda_2} + 2\sqrt{\frac{2}{\pi}} \Big)\biggr) =\mathcal{O}\Big(\KK^{-\frac{1}{3}} \sigma^{\frac{2}{3}}\Big)\,.
\end{align*}
An explicit sufficient lower bound for $\KK$ in terms of $\sigma_0$ can be computed from the two conditions
\begin{align}\label{eq_KK_condition}\nonumber
1 - \frac{2}{\pi} & > \Bigg(2 R_{\Lambda_1} + R_{\Lambda_2} \Bigg(R_{\Lambda_2} + 2 \sqrt{\frac{2}{\pi}}\Bigg)\Bigg)+ \sigma \bigl(R_{\partial \Lambda_1} + R_{\partial \Lambda_2} \bigr),\\
1 - \sqrt{\frac{2}{\pi}} & >  R_{\Lambda_1} + R_{\Lambda_2} \Bigg(R_{\Lambda_2} + 2\sqrt{\frac{2}{\pi}} \Bigg).
\end{align}
It remains to show the boundedness property for the first two derivatives of $\tilde{\bf \Psi}$. By Lemma \ref{lem_J_is_analytic}, we have
\begin{align}\label{eq_lem_J_is_analytic_0}
\big|\partial^k J\bigl(\sigma\bigr)\big| \leq C \left(1+\sigma^{-k}\right),~k=1,2,
\end{align}
where \(J\bigl( \sigma\bigr) = 2 \Lambda_1(\sigma) - \Lambda_2^2(\sigma)\) and $C$ is a constant not depending on $\sigma$. Observe that
\begin{align*}
\partial^2\tilde{\bf \Psi} = 4 J + 6 \sigma \partial J + \sigma^2 \partial^2J,
\end{align*}
hence the claim follows.
\end{proof}

\begin{proof}[Proof of Proposition \ref{prop_psi_emp}]
The proof can be redirected to Proposition \ref{thm_psi_properties} using Proposition \ref{prop_min_distrib_det} and a truncation argument for the integrals over $x$. The corresponding computations are very similar to those above and the Lemmas given below. We therefore omit the details.
\end{proof}

\begin{lem}\label{lem_xp}
For $\KK > 0$, $p \in \N_0$, we obtain the following decay behaviour of the moment integrals:
\begin{align*}
\int_0^{\infty}x^p \E\left[e^{-\KK \sigma H(x)}\right]dx  \leq 2^{p+1}\Bigg(\frac{\E[\abs{Z}^{p+1}]}{p+1}+\Gamma(p + 1)\Bigg)(\KK\sigma)^{-(p+1)/3}
\end{align*}
with $Z\sim N(0,1)$.

\begin{comment}
\begin{align*}
\int_0^{\infty}x^p \E\left[e^{-\KK \sigma H(x)}\right]dx  \leq 3\frac{ 4^{p/2}}{ \sqrt{2\pi}} l^{\frac{p+1}{2}} \Gamma(p+1) + \frac{2}{\KK \sigma l}\bigl(1 + \Gamma\bigl(p + 1\bigr)\bigr)\,.
\end{align*}
\end{comment}

\end{lem}

\begin{proof}[Proof of Lemma \ref{lem_xp}]
The following useful relation in terms of the $N(0,1)$-distribution function $\Phi$ is derived from the law of the minimum of Brownian motion:
\begin{align}\label{eq_bound_Tx}
\P\bigl(T_x \leq l\bigr) &= 2(1-\Phi(\abs{x}/\sqrt{l})).
\end{align}
Then for $0 < l < 1$
\begin{align*}
\int_0^{\infty}x^p \E\left[e^{-\KK \sigma H(x)}\right]dx &\leq  \int_{0}^{\infty}x^p \E\left[\1\bigl(T_{-x/2} \leq l \bigr)\right]dx \\&+ \int_{0}^{\infty}x^p \E\left[e^{-\KK \sigma H(x)}\1\bigl(T_{-x/2} > l \bigr)\right]dx \stackrel{def}{=} R_1 + R_2\,.
\end{align*}

By \eqref{eq_bound_Tx} we have
\begin{align*}
R_1 &= 2\int_0^\infty x^p(1-\Phi(x/\sqrt{4l}))\,dx\\
&=(4l)^{(p+1)/2}\int_0^\infty 2z^p(1-\Phi(z))\,dz=(4l)^{(p+1)/2}\frac{\E[\abs{Z}^{p+1}]}{p+1}.
\end{align*}

We further note that $T_{-x/2}>l$ implies $H(x)\ge\int_0^l(-x/2+x)_+dt=lx/2$, such that
\begin{align*}
R_2 &\leq \int_{0}^{\infty}x^p e^{-\KK \sigma l x/2}dx =\left(\frac{2}{\KK \sigma l}\right)^{p + 1} \Gamma\bigl(p + 1\bigr).
\end{align*}
Choosing $l=(\KK\sigma)^{-2/3}$, we obtain
\begin{align*}
\int_0^{\infty}x^p \E\left[e^{-\KK \sigma H(x)}\right]dx  \leq \Bigg(2^{p+1}\frac{\E[\abs{Z}^{p+1}]}{p+1}+2^{p+1}\Gamma\bigl(p + 1\bigr)\Bigg)(\KK\sigma)^{-(p+1)/3},
\end{align*}
as asserted.
\end{proof}

\begin{lem}\label{lem_HK}
Let $\KK > 0$. Then
\begin{align*}
\int_0^{\infty}\E\left[\KK H(x)e^{-\KK \sigma H(x)}\right]dx  &\leq  4\big(1+1/\sqrt{2\pi}\big)\KK^{-1/3}\sigma^{-4/3}.
\end{align*}
\begin{comment}
\begin{align*}
\biggl|\int_0^{\infty}\E\left[\KK H(x,\ad)e^{-\KK \sigma H(x,\ad)}\right]dx \biggr| &\leq
\delta + \frac{4}{\sigma \sqrt{2 \pi}} \sqrt{l} |\log(\sqrt{l} \delta)| + 4\sqrt{l} + 3^{\ad}  e^{-\KK \sigma l (\delta/2)^{\ad}} \\&+ \KK^{-\ad} \frac{4^{\ad}}{(\sigma l 2^{-\ad})^{\ad + 1}} \Gamma(\ad + 1)\\&\leq \frac{8\sqrt{l}\log l}{\sigma} + \frac{2^{\ad^2 + 3\ad +1}\Gamma(\ad + 1)}{\KK \sigma^{\ad +1} l^{3/2(\ad + 1)}}.
\end{align*}
use $l^{1/4} \log l \leq 3/2$
\end{comment}
\end{lem}

\begin{proof}[Proof of Lemma \ref{lem_HK}]
We make the decomposition
\begin{align*}
\int_{0}^{\infty}\E\left[\KK \sigma H(x)e^{-\KK \sigma H(x)}\right]dx &= \int_{0}^{\infty}\E\left[\1(T_{-x/2} \leq l )\KK\sigma H(x)e^{-\KK \sigma H(x)}\right]dx \\&+ \int_{0}^{\infty}\E\left[\1(T_{-x/2} > l )\KK\sigma H(x)e^{-\KK \sigma H(x)}\right]dx\, ,
\end{align*}
with some $l>0$.
Using  $y e^{-y} \leq 1$ and \eqref{eq_bound_Tx}, we obtain
\begin{align*}
\int_{0}^{\infty}\E\left[\1(T_{-x/2} \leq l )\KK\sigma H(x)e^{-\KK \sigma H(x)}\right]dx&\le  2\int_0^\infty(1-\Phi(x/\sqrt{4l}))\,dx\\
&=\sqrt{8l/\pi}.
\end{align*}
Now, using $ye^{-y}\le e^{-y/2}$ and $T_{-x/2}>l\Rightarrow H(x)\ge lx/2$, we bound the other term by
\[ \int_{0}^{\infty}\E\left[\1(T_{-x/2} > l )\KK\sigma H(x)e^{-\KK \sigma H(x)}\right]dx \leq \int_0^\infty e^{-\KK \sigma lx/4}dx=\frac{4}{\KK \sigma l}.\]
The choice $l=(\KK\sigma)^{-2/3}$ and division by $\sigma$ yield the claim.
\begin{comment}
Observe that
\begin{align}\label{eq_bound_IW}\nonumber
\E\left[\int_0^1 (W_s + x)_+\right] &\leq \int_0^1\biggl(\E\bigl[|W_s|\bigr] + x \biggr)ds \leq \int_0^1 \bigl(2\sqrt{s} + x\bigr) ds \\&\leq \bigl(2 + x\bigr).
\end{align}
Moreover, since $\Gamma(p) = \int_0^{\infty} x^p e^{-x} dx$, we obtain
\begin{align*}
&\int_{\delta}^{\infty}\E\left[\1(T_x > l )\KK H(x)e^{-\KK \sigma H(x)}\right]dx \leq \int_{\delta}^{\infty}\E\left[\KK H(x)e^{-\KK \sigma l (x/2)^{}}\right]dx \\
&\quad \leq \int_{\delta}^{\infty}\KK \bigl(2 + x\bigr)^{} e^{-\KK \sigma l (x/2)^{}}dx \leq \int_{\delta}^{1}\KK \bigl(2 + x\bigr)^{} e^{-\KK \sigma l (x/2)^{}}dx\\&\quad + 4^{} \int_{1}^{\infty}\KK x^{} e^{-\KK \sigma l (x/2)^{}}dx  \leq 3^{}  e^{-\KK \sigma l (\delta/2)^{}} + \frac{4^{} \KK}{(\sigma l 2^{-1} \KK)^{2}} \Gamma(2).
\end{align*}
The claim now follows by balancing all the above bounds, appropriately selecting $\delta, l > 0$.
\end{comment}
\end{proof}

\begin{lem}\label{lem_HK_x}
Let $\KK > 0$. Then
\begin{align*}
\int_0^{\infty}x\E\left[\KK H(x)e^{-\KK \sigma H(x)}\right]dx &\leq \frac{6}{\KK^{1/2} \sigma^{3/2} }.
\end{align*}
\end{lem}

\begin{proof}[Proof of Lemma \ref{lem_HK_x}]
We proceed as in the proof of Lemma \ref{lem_HK} and obtain for any $l>0$
\begin{align*}
 \int_0^{\infty}x\E\left[\KK\sigma H(x)e^{-\KK \sigma H(x)}\right]dx &\leq \int_0^\infty x\Big(2(1-\Phi(x/\sqrt{4l}))+e^{-\KK\sigma lx/4}\Big)\,dx\\
 &= 2l+(\KK\sigma l/4)^{-1}.
\end{align*}
The result follows with $l=(\KK\sigma)^{-1/2}$.
\end{proof}

\begin{lem}\label{lem_bound_I_KK_plus_derivative}
Let $\KK \ge \sigma^{-1}$. Then
\begin{align*}
\int_0^1 I(\KK\sigma,s)ds &\leq 3\Big(\frac{\KK\sigma}{\log(\KK\sigma)}\Big)^{-2/5},&
\biggl|\frac{\partial \int_0^1 I(\KK\sigma,s)ds}{\partial \sigma}\biggr| &\leq  \frac{3}{\sigma}\Big(\frac{\KK\sigma}{\log(\KK\sigma)}\Big)^{-2/5}.
\end{align*}
\begin{comment}
\begin{align*}
\biggl|\frac{\partial \int_0^1 I(\KK\sigma, \ad,s)ds}{\partial \sigma}\biggr| &\leq \frac{\delta}{\sigma} + \frac{2}{\ad +2} \frac{1}{\sigma \sqrt{\delta}}\frac{\sqrt{2}}{\sqrt{\pi}} l + \frac{\bigl(\KK \sigma l + 1\bigr)}{(l \sigma)^2} \exp(- \KK \delta^{1 + \ad/2} l \sigma)\\&\leq \frac{l^{2/3}}{\sigma} + \frac{4}{\KK \sigma^3 l^{5 + \ad}},
\end{align*}
provided that $\KK\delta^{1+ \ad/2} \geq 1$.
\end{comment}
\end{lem}

\begin{proof}[Proof of Lemma \ref{lem_bound_I_KK_plus_derivative}]
With $\lambda(s)=\KK(1-s)^{3/2}$ we obtain for any $T>0$
\begin{align*}
\int_0^1 I(\KK\sigma,s)ds&=\int_0^1\E\Big[  e^{-\lambda(s)\sigma \int_0^1(W_t)_-dt}\Big]\,ds\\
&\le \int_0^1\Big(\P\Big(\lambda(s)\sigma\int_0^1(W_t)_-dt\le T\Big)+e^{-T}\Big)\,ds.
\end{align*}
From $\int_0^1(W_t)_-\,dt\ge \abs{Z}$ with $Z=\int_0^1W_t\, dt\sim N(0,1/3)$, we deduce $\P(\int_0^1(W_t)_-dt\le \eps)\le \eps$, $\eps>0$, and thus
\[ \babs{\frac{\partial \int_0^1 I(\KK\sigma,s)ds}{\partial \sigma}}\le\int_0^1\Big(\Big(T\sigma^{-1}\lambda(s)^{-1}\Big)\wedge 1\Big)\,ds+e^{-T}.
\]
Using $(\sigma\lambda(s))^{-1}\le (\KK\sigma/T)^{3/5}$ for $s\le 1-(\KK\sigma/T)^{-2/5}$, the last integral is bounded by $2(\KK\sigma/T)^{-2/5}$. The choice $T=\log (\KK\sigma)$ yields the first inequality.

Then using $ye^{-y}\le e^{-y/2}$ we also obtain
\begin{align*}
\babs{\frac{\partial \int_0^1 I(\KK\sigma,s)ds}{\partial \sigma}}&=\int_0^1\E\Big[ \lambda(s)\int_0^1(W_t)_-dt\, e^{-\lambda(s)\sigma \int_0^1(W_t)_-dt}\Big]\,ds\\
&\le \sigma^{-1}\int_0^1\Big(\P\Big(\lambda(s)\sigma\int_0^1(W_t)_-dt\le T\Big)+e^{-T/2}\Big)\,ds.
\end{align*}
The previous bounds now apply in the same way.
\end{proof}

%\MR: naechstes Lemma auf zweite Ableitung reduzieren!

\begin{lem}\label{lem_J_is_analytic}
Consider
\(
J\bigl(\sigma\bigr) = 2 \Lambda_1(\sigma) - \Lambda_2^2(\sigma)\).
Then there exists a constant $B = B(\KK) > 0$ only depending on $\KK$ such that
\begin{align}\label{eq_lem_J_is_analytic_1}
\big|\partial^k J\bigl(\sigma\bigr)\big| \leq  B(1+\sigma^{-k}), \quad k = 1,2.
\end{align}
\end{lem}
%Note that \KK contains \lambda
\begin{proof}[Proof of Lemma \ref{lem_J_is_analytic}]
Without loss of generality, we may assume that $\KK = 1$. From the considerations below, existence of the $k$'th derivative of $J\bigl(\sigma\bigr)$ with respect to $\sigma$ follows. We thus focus on establishing \eqref{eq_lem_J_is_analytic_1}. First consider $\int_0^{\infty}x\E\left[e^{-\sigma H(x)}\right]dx$. An application of the Cauchy-Schwarz inequality gives
\begin{align}\nonumber \label{eq_lem_J_is_analytic_2}
\biggl|\frac{\partial^k{\int_0^{\infty}x\E\left[e^{-\sigma H(x)}\right]dx}}{\partial^k\sigma}\biggr| &= \biggl|\int_0^{\infty}x\E\left[ \bigl(-  H(x)\bigr)^k e^{-\sigma H(x)}\right]dx\biggr|  \\&\leq \int_0^{\infty}x\E\left[ H(x)^{2k} \right]^{1/2} \E\left[e^{-2\sigma H(x)}\right]^{1/2}dx.
\end{align}
Applying the triangle and Cauchy-Schwarz inequality further yields
\begin{align}\label{eq_lem_J_is_analytic_5}
\E\left[ H(x)^{2k} \right]^{1/2} &\leq \E\left[ \biggl(\int_0^{1} |W_s|^{}ds + |x|\biggr)^{2k}\right]^{1/2} \lesssim 1 \vee x^k.
\end{align}
The calculations in the proof of Lemma \ref{lem_xp} with $l=\sqrt{x/\sigma}/2$ yield
\begin{equation} \E[\exp(-2\sigma H(x))]\lesssim \exp(-x^{3/2}\sigma^{1/2}/2).\label{eq_lem_J_is_analytic_6}\end{equation}
Combining \eqref{eq_lem_J_is_analytic_5} and \eqref{eq_lem_J_is_analytic_6}, we deduce that
\begin{align}\nonumber
&\int_0^{\infty}x\E\left[ H(x)^{2k} \right]^{1/2} \E\left[e^{-2\sigma H(x)}\right]^{1/2}dx \\& \nonumber \lesssim \int_0^{\infty} \bigl(x \vee x^{k + 1} \bigr) \exp(-x^{3/2}\sigma^{1/2}/2)\,dx \lesssim \sigma^{-2/3}(1+\sigma^{-k/3}).
\end{align}
This implies that for some $C> 0$
\begin{align}
\biggl|\frac{\partial^k\int_0^{\infty}x\E\left[e^{-\sigma H(x)}\right]dx}{\partial^k\sigma}\biggr| \le  C\sigma^{-2/3}(1+\sigma^{-k/3})\,.
\end{align}
Arguing in the same manner, one also establishes that
\begin{align}
\biggl|\frac{\partial^k\int_0^{\infty}\E\left[e^{-\sigma H(x)}\right]dx}{\partial^k\sigma}\biggr| \le C\sigma^{-1/3}(1+\sigma^{-k/3})\,.
\end{align}
Moreover, such bounds are also valid for the derivatives of $\int_0^1 I(\sigma,s)ds$.
\end{proof}

\section{Proof of Theorem 2}\label{prooflb}
After the reductions of the problem to a simpler and more informative experiment,
we now prove Theorem \ref{theolower} using properties of the Hellinger distance $H(P,Q)$ between probability
measures, in particular $H^2(P_1\otimes P_2,Q_1\otimes Q_2)\le
H^2(P_1,Q_1)+H^2(P_2,Q_2)$ (subadditivity under independence),
$H^2(P,Q)=\E[H^2(P,Q|T)]$ (Hellinger distance conditional on a statistic $T$) and
\[H^2(PPP(\lambda_1),PPP(\lambda_2))\le \int (\sqrt{\lambda_1}-\sqrt{\lambda_2})^2\]
(Hellinger bound for PPP measures with intensity densities $\lambda_i$, cf. \cite{kutoyants}).

    Put $\delta_n=\delta \sigma_0^{5/3}n^{-1/3}$. From
$H^2(N(0,\sigma_0^2),N(0,\sigma_0^2+\delta_n))\le 2(\delta_n\sigma_0^{-2})^2$,
cf.\,Appendix in \cite{reiss}, and the independent increments of Brownian
motion we infer for the Hellinger distance of the laws of
$(X_{T_j^s})_{j=1,\ldots,J}$ under $\sigma_0^2$ and $\sigma_0^2+\delta_n$
    \[
H^2\Bigg(P^{(X_{T_j^s})}_{\sigma_0^2},P^{(X_{T_j^s})}_{\sigma_0^2+\delta_n}\,\big|\,(T_j^s)\Bigg)\le
\sum_{j=1}^J 2\delta_n^2\sigma_0^{-4}=2J\delta_n^2\sigma_0^{-4}.\]
    For each PPP with intensity density $\lambda^j$ we obtain by integral calculations, in
terms of $\eta=(\sigma_0^2+\delta_n)^{1/2}-\sigma_0$:
    \begin{align*}
     &H^2\Big(PPP(\lambda^j(\sigma_0^2)),PPP(\lambda^j(\sigma_0^2+\delta_n))\,\big|\,(T_j^s),B^{0,T_j^s-T_{j-1}^s}\Big)\\
     &\le n\int_0^{T_j^s-T_{j-1}^s}\int_{\R} \Big(\Big(\h^{-1}\big(y-\sigma_0
B^{0,T_j^s-T_{j-1}^s}_t\big)_+\Big)\wedge 1\\
      &\qquad -\Big(\h^{-1}\big(y-(\sigma_0^2+\delta_n)^{1/2}
B^{0,T_j^s-T_{j-1}^s}_t\big)_+\Big)\wedge 1\Big)^2dy\,dt\\
     &= n\h\int_0^{T_j^s-T_{j-1}^s}\int_{\R} \Big(u_+\wedge 1 -\big(u- \h^{-1}\eta
B^{0,T_j^s-T_{j-1}^s}_t\big)_+\wedge 1\Big)^2dy\,dt\\
     &\le n\h\int_0^{T_j^s-T_{j-1}^s}\h^{-2}\eta^2(B^{0,T_j^s-T_{j-1}^s}_t)^2dt.
     \end{align*}
     Hence, by using the variance of a Brownian bridge we arrive at
     \begin{align*}
     &H^2\Big(PPP(\lambda^j(\sigma_0^2)),PPP(\lambda^j(\sigma_0^2+\delta_n))\,\big|\,(T_j^s)\Big)\\
     &\le
n\h^{-1}\eta^2\int_0^{T_j^s-T_{j-1}^s}t(1-(T_j^s-T_{j-1}^s)^{-1}t)\,dt=\frac{n\eta^2}{6\h}(T_j^s-T_{j-1}^s)^2.
    \end{align*}
    Since conditional on $(T_j^s)$ all observations are independent, the total
squared Hellinger distance conditional on $(T_j^s)$ is bounded by
    \[
2J\delta_n^2\sigma_0^{-4}+\frac{n\eta^2}{6\h}\sum_{j=1}^{J+1}(T_j^s-T_{j-1}^s)^2.\]
    Taking expectations and using $J\sim \Poiss(2n\h/3)$, $T_j^s-T_{j-1}^s\sim
\Exp(2n\h/3)$ to apply the Wald identity to the second sum,  the unconditional total Hellinger distance is bounded by
    \[ H^2\le\frac{4n\h\delta_n^2}{3\sigma_0^4}+\frac{n\eta^2}{6\h} (2n\h/3)^{-1}(1+o(1)).\]
    We have $\eta^2\le \frac12 \delta_n^2 \sigma_0^{-2}$ due to $\sqrt{1+x}\le
1+x/2$ for $x>0$ and thus by choosing $\h\propto (\sigma_0^2/n)^{1/3}$
optimally and plugging in $\delta_n$
    \[ H^2\le  \delta^2n^{1/3}\sigma_0^{4/3}
\inf_{\h>0}\Big(\frac{2\h}{3\sigma_0^2}+\frac{C}{12\h^2n}\Big)\le C'\delta^2.\]

From the general lower bound Theorem 2.2(ii) in \cite{tsybakov}  we thus obtain the result if $\delta$ is
chosen smaller than $2/C'$.

%%%%%%%%%%%%%%%%%%%%%%%%%%%%%%%%%%%%%%%%%%%%%%%%%%%%%%%%%%%%%%%%%%%%%%%%%%%%%%%%%%%%%%%%%%%%%%%%%%%%%%%%%%%%%%%%%%%%%%%%%%%%%%%%%%%%%%%%%%%%%%%%%%%%%%%%%%%%%%%%%%%%%%%%

\end{document}